\documentclass{aic}

\usepackage{tikz}
\usetikzlibrary{math}
\usetikzlibrary{calc,decorations.pathmorphing,through}
\tikzset{snake it/.style={decorate, decoration=snake}}

\definecolor{DarkDesaturatedBlue}{HTML}{3A3556}
\definecolor{VividOrange}{HTML}{F15918}
\definecolor{PureOrange}{HTML}{FFBA00}
\definecolor{LightGrayishPink}{HTML}{EEC5D5}
\definecolor{VerySoftBlue}{HTML}{B5AFDB}

\newcommand{\triple}[7]{
	
	\ifx\relax#4\relax
	\def\qoffs{0pt}
	\else
	\def\qoffs{#4}
	\fi
	
	\def\qhedge{
		($#1+#3!\qoffs!-90:#2-#3$) --
		($#2+#1!\qoffs!-90:#3-#1$) --
		($#3+#2!\qoffs!-90:#1-#2$) -- cycle}
	
	\coordinate (12) at ($#1!\qoffs!90:#2$);
	\coordinate (13) at ($#1!\qoffs!-90:#3$);
	\coordinate (23) at ($#2!\qoffs!90:#3$);
	\coordinate (21) at ($#2!\qoffs!-90:#1$);
	\coordinate (31) at ($#3!\qoffs!90:#1$);
	\coordinate (32) at ($#3!\qoffs!-90:#2$);
	
	\def\nqhedge{
		(13) let \p1=($(13)-#1$), \p2=($(12)-#1$) in
		arc[start angle={atan2(\y1,\x1)}, delta angle={atan2(\y2,\x2)-atan2(\y1,\x1)-360*(atan2(\y2,\x2)-atan2(\y1,\x1)>0)}, x radius=\qoffs, y radius=\qoffs] --
		(21) let \p1=($(21)-#2$), \p2=($(23)-#2$) in
		arc[start angle={atan2(\y1,\x1)}, delta angle={atan2(\y2,\x2)-atan2(\y1,\x1)-360*(atan2(\y2,\x2)-atan2(\y1,\x1)>0)}, x radius=\qoffs, y radius=\qoffs] --
		(32) let \p1=($(32)-#3$), \p2=($(31)-#3$) in
		arc[start angle={atan2(\y1,\x1)}, delta angle={atan2(\y2,\x2)-atan2(\y1,\x1)-360*(atan2(\y2,\x2)-atan2(\y1,\x1)>0)}, x radius=\qoffs, y radius=\qoffs] --
		cycle}
	
	\ifx\relax#5\relax
	\def\qlwidth{1pt}
	\else
	\def\qlwidth{#5}
	\fi
	
	\ifx\relax#7\relax
	\fill \nqhedge;
	\else
	\fill[#7]\nqhedge;
	\fi
	
	\ifx\relax#6\relax
	\draw[line width=\qlwidth,rounded corners=\qoffs]\nqhedge;
	\else
	\draw[line width=\qlwidth,#6]\nqhedge;
	\fi
}

\usepackage{cleveref}
\theoremstyle{plain}
\newtheorem{theorem}{Theorem}[section]
\crefname{theorem}{Theorem}{Theorems}

\newtheorem{proposition}[theorem]{Proposition}
\crefname{proposition}{Proposition}{Propositions}

\newtheorem{corollary}[theorem]{Corollary}
\crefname{corollary}{Corollary}{Corollaries}

\newtheorem{lemma}[theorem]{Lemma}
\crefname{lemma}{Lemma}{Lemmas}

\crefname{conjecture}{Conjecture}{Conjectures}

\crefname{problem}{Problem}{Problem}

\crefname{claim}{Claim}{Claims}

\crefname{observation}{Observation}{Observations}

\crefname{setup}{Setup}{Setups}

\crefname{myth}{Myth}{Myths}

\crefname{fact}{Fact}{Facts}

\crefname{algorithm}{Algorithm}{Algorithms}

\crefname{example}{Example}{Examples}

\theoremstyle{definition}
\newtheorem{definition}[theorem]{Definition}
\crefname{definition}{Definition}{Definitions}

\crefname{construction}{Construction}{Constructions}

\crefname{question}{Question}{Questions}

\newtheorem{remark}[theorem]{Remark}
\crefname{remark}{Remark}{Remarks}

\numberwithin{equation}{section}

\usepackage[shortlabels]{enumitem}
\setlist[enumerate,1]{label={\upshape (\roman*)}}

\def\eps{\varepsilon}
\DeclareMathOperator{\probability}{Pr}
\DeclareMathOperator{\expectation}{\mathbf{E}}

\DeclareMathOperator{\treewidth}{tw}
\DeclareMathOperator{\orderedshadow}{\partial^\circ}
\newcommand{\sta}[1]{\mathop{\mkern 0mu\mathrm{sta}}\nolimits(#1)}
\newcommand{\ter}[1]{\mathop{\mkern 0mu\mathrm{ter}}\nolimits(#1)}
\renewcommand{\int}[1]{\mathop{\mkern 0mu\mathrm{int}}\nolimits(#1)}
\newcommand{\clique}[2]{\smash{K_{#2}^{#1}}}

\aicAUTHORdetails{%
  title = {Towards a  Hypergraph version of the P\'osa--Seymour Conjecture}, 
  author = {Mat\'ias Pavez-Sign\'e, Nicol\'as Sanhueza-Matamala, and Maya Stein},
  plaintextauthor = {Matias Pavez-Signe, Nicolas Sanhueza-Matamala, Maya Stein},
  plaintexttitle = {Towards a Hypergraph version of the Posa-Seymour Conjecture}, 
  runningtitle = {Towards a  Hypergraph version of the P\'osa--Seymour Conjecture}, 
  runningauthor = {Mat\'ias Pavez-Sign\'e, Nicol\'as Sanhueza-Matamala, and Maya Stein},
  keywords = {Extremal hypergraph theory, Hamilton cycles, Hypertrees.},
}   

\aicEDITORdetails{%
   year={2023},
   number={3},
   received={11 January 2023},   
   published={22 July 2023},  
   doi={10.19086/aic.2023.3}      
}   

\begin{document}

\begin{frontmatter}[classification=text]

\title{Towards a  Hypergraph version of the P\'osa--Seymour Conjecture} 

\author[mps]{Mat\'ias Pavez-Sign\'e\thanks{Supported by the European Research Council grant 947978 under
		the European Union's Horizon 2020 research and innovation programme, and by a CMM postdoctoral scholarship funded by ANID PIA AFB170001 while he was at the Universidad de Chile.}}
\author[nsm]{Nicol\'as Sanhueza-Matamala\thanks{Supported by the Czech Science Foundation, grant number GA19-08740S with institutional support RVO: 67985807 while he was affiliated with the Institute of Computer Science of the Czech Academy of Sciences; and by ANID-Chile through the FONDECYT Iniciación Nº11220269 grant.}}
\author[ms]{Maya Stein\thanks{Supported by FONDECYT Regular Grant 1221905, by MathAmSud MATH190013, by FAPESP-ANID Investigaci\'on Conjunta grant 2019/13364-7, and by ANID PIA CMM FB210005.}}

\begin{abstract}
We prove that for fixed $r\ge k\ge 2$, every $k$-uniform hypergraph on $n$ vertices having minimum codegree at least $(1-(\binom{r-1}{k-1}+\binom{r-2}{k-2})^{-1})n+o(n)$
contains the $(r-k+1)$th power of a tight Hamilton cycle.
This result may be seen as a step towards a hypergraph version of the P\'osa--Seymour conjecture.

Moreover, we prove that the same bound on the codegree suffices for finding a copy of every spanning hypergraph of tree-width less than $r$ which admits a tree decomposition where every vertex is in a bounded number of bags.
\end{abstract}
\end{frontmatter}

\section{Introduction}

\subsection{Powers of tight Hamilton cycles}
A central problem in extremal graph theory is the study of degree conditions that force a graph $G$ to contain a copy of some large or even spanning subgraph. A classical result in this area is Dirac's theorem~\cite{D1952} from the 1950's which states that every graph $G$ on $n\ge 3$ vertices with minimum degree at least $ n/2$ contains a Hamilton cycle, i.e.~a cycle that passes through all vertices of $G$.

Over the decades to follow, there have been numerous efforts to extend Dirac's theorem to other spanning subgraphs. One of the most important extensions concerns \emph{$r$th powers of Hamilton cycles} (i.e.~Hamilton cycles with an additional chord for each pair of vertices at distance at most $r$).
P\'osa (for $r=3$, see~\cite{problem9}) and Seymour (for $r > 3$~\cite{seymour73}) conjectured that for all $r \geq 3$ and all $n \geq r$,
every $n$-vertex graph $G$ with $\delta(G) \geq (r-1)n/r$ must contain the $(r-1)$th power of a Hamilton cycle. This  conjecture has been confirmed
for every $r \geq 3$ and for all large $n \geq n_0(r)$ by Koml\'os, S\'ark\"ozy and Szemer\'edi~\cite{KSS1998b}.

Many extremal results for graphs have found extensions to hypergraphs, in particular to $k$-uniform hypergraphs, which we will call {\em $k$-graphs} for short.  In a $k$-graph $H$, the minimum degree condition is often translated to a condition on the \emph{minimum codegree} $\delta(H)$, which is  defined as the maximum number $m$ such that each set of $k-1$ vertices lies in at least $m$ edges of $H$.
For $j \ge k\geq 1$, the \emph{$j$th power of a $k$-cycle} on $n$ vertices, denoted by $C^{j}_{k, n}$, is the $k$-graph on $n$ cyclically ordered vertices that contains every $k$-subset of every set of $k+j-1$ consecutive vertices. The $k$-graph $C^1_{k,n}$ is also called a {\em tight Hamilton cycle}.

Extending Dirac's theorem to this setting, 
R\"odl, Ruci\'nski and Szemer\'edi~\cite{RRS2008} showed that any $n$-vertex $k$-graph $H$ with $\delta(H) \geq (1/2 + o(1))n$ contains a tight Hamilton cycle. For powers of tight hypergraph cycles, not much is known in general, but 
Bedenknecht and Reiher~\cite{BR2020} showed recently that any  $3$-graph on $n$ vertices with codegree at least $( 4/5 + o(1))n$ contains a copy of $C^2_{3,n}$.

Our first main result establishes a generalisation of  these results to arbitrary powers of $k$-cycles and arbitrary $k\ge 2$.

\begin{theorem} \label{theorem:main-powcycles}
	For all $r\ge k\ge 2$ and $\alpha > 0$, there exists $n_0$ such that   every $k$-graph $H$ on $n\ge n_0$ vertices with 
	\[ \delta(H) \geq \left(1 - \frac{1}{\binom{r-1}{k-1} + \binom{r-2}{k-2}} + \alpha \right)n. \]
	contains a copy of $C^{r-k+1}_{k,n}$.
\end{theorem}

Let us compare this result with what one can get from previously established results.
Given $d < k$ and a $k$-graph $H$, the \emph{minimum $d$-degree} $\delta_d(H)$ is the maximum number $m$ such that each set of $d$ vertices lies in at least $m$ edges of $H$.
Lang and the second author~\cite{LangSanhuezaMatamala2022} showed that $n$-vertex $k$-graphs $H$ with $\delta_d(H) \geq (1 - 1/(2(k-d)) + o(1)) \binom{n}{k-d}$ contain tight Hamilton cycles.
By applying this result to the $r$-uniform clique graph (it will be defined in \Cref{def:clique-graph}) and $d = k-1$,
a short computation yields that $n$-vertex $k$-graphs with $\delta(H) \geq (1 - 1/(2(r-k+1) \binom{r}{k}) + o(1))n$ contain a copy of $C^{r-k+1}_{k,n}$.
\cref{theorem:main-powcycles} yields the same outcome with a substantially lower minimum codegree condition.

For $k = 2$, and any $r\ge 2$, \cref{theorem:main-powcycles} 
gives an asymptotic version of the P\'osa--Sey\-mour conjecture, as shown by Koml\'os, S\'ark\"ozy and Szemer\'edi~\cite{KSS1998} before their proof of the exact version in~\cite{KSS1998b}. For $r = k\ge 2$, the degree condition of \cref{theorem:main-powcycles}  becomes $\delta(H) \geq (1/2 + o(1))n$, and we recover the R\"odl, Ruci\'nski and Szemer\'edi's Dirac-type threshold for tight Hamilton cycles~\cite{RRS2008}. For $r=k+1\ge 3$,  \cref{theorem:main-powcycles} yields the following corollary.
\begin{corollary}For all $k\ge 2$ and $\alpha>0$, there exists $n_0$ such that every $k$-graph $H$ on $n\ge n_0$ vertices with $\delta(H)\ge(1-1/({2k-1}) + \alpha)n$ contains the square (i.e.~second power) of a tight Hamilton cycle.
\end{corollary}
For $k=3$, this equals the above-cited result by Bedenknecht and Reiher~\cite{BR2020}.
As in their proof, we will also rely on `absorption' techniques combined with `connection' steps. For an  outline of our proof see \Cref{section:proofoverview}.

\subsection{Hypergraphs of bounded tree-width}
With the same conditions as in \cref{theorem:main-powcycles}, we can also embed a larger family of $k$-graphs, namely those of sufficiently bounded \emph{tree-width}, which we define now. For a hypergraph $H$ and $j \geq 1$, the \emph{$j$th shadow of $H$} is the $j$-graph $\partial_j(H)$ on $V(H)$ where a $j$-set $X$ is an edge if and only if it is contained in some edge of $H$.
An $r$-uniform graph $T$ is an \emph{$r$-tree} if
\begin{enumerate}
	\item $T$ consists of a single edge, or
	\item there exists an $r$-tree $T'$,
	a set $X \in \partial_{r-1}(T')$,
	and a vertex $v \notin V(T')$ such that
	$T$ is obtained from $T'$ by adding $v$ and the edge $X \cup \{v\}$.
\end{enumerate}
The tree-width of a graph is a parameter measuring how closely the graph resembles  a tree.
One of the usual definitions of the tree-width
uses the notion of `partial $k$-trees', which is equivalent to  defining  the \emph{tree-width} $\treewidth(G)$ of a graph $G$ 
as the minimum $r$ such that there is an $(r+1)$-tree $T$ with $G \subseteq \partial_2(T)$
(see, e.g.~\cite[Theorem 35]{B1998}, for this and  equivalent definitions of tree-width).
For $k\ge 3$, define the \emph{tree-width $\treewidth(G)$} of a $k$-graph $G$ as  the minimum $r$ such that there is an $(r+1)$-tree $T$ with $G \subseteq \partial_k(T)$, in which case we  say that $G$ \emph{admits}~$T$.

With these definitions at hand, we can state our second main result.

\begin{theorem} \label{theorem:main-treewidth}
	For any $r\ge k\ge 2$,  $\Delta > 0$ and $\alpha > 0$, there exists $n_0$ such that the following holds for all $n \geq n_0$.
	Let $G$ be an $n$-vertex $k$-graph 
	with $\treewidth(G)<r$,
	which admits an $r$-tree $T$ with $\Delta_1(T) \leq \Delta$, and let $H$ be an $n$-vertex $k$-graph  with 
	\[ \delta(H) \geq \left(1 - \frac{1}{\binom{r-1}{k-1} + \binom{r-2}{k-2}} + \alpha \right)n. \]
	Then $G \subseteq H$.
\end{theorem}
For $r\ge k= 2$,
the codegree condition of
\cref{theorem:main-treewidth} is $\delta(H) \geq (1-1/r + o(1))n$, and the 
graph $G$ needs to  admit an $r$-tree $T$ of bounded degree.
It is not difficult to see that any such  $G$ has bounded-degree and  is $(r-1)$-degenerate, and
thus has  chromatic number at most $r$.
Also, graphs of bounded degree and tree-width have sublinear bandwidth~\cite{BPTW2010},
so our result recovers a particular instance of the Bandwidth Theorem by B\"ottcher, Schacht and Taraz~\cite{BST2009}.

For $r=k\ge 2$, the codegree condition of
\cref{theorem:main-treewidth} becomes $\delta(H) \geq ( 1/2 + o(1))n$, and the $k$-graph $G$ needs to admit a $k$-tree $T$ of bounded degree, which is equivalent to $G$ itself being a bounded-degree $k$-tree.
For $r = k = 2$ this is a result by Koml\'os, S\'ark\"ozy and Szemer\'edi~\cite{KSS1995} on embedding spanning bounded-degree trees, and for $r = k > 2$ we recover a previous result by the authors~\cite{PSS2020} on the embedding of spanning bounded-degree $k$-trees.

\subsection{Lower bounds and clique tilings} \label{section:lowerbounds}

Most of the results implied by our two theorems are asymptotically tight, that is, the codegree conditions in~\cite{BST2009,KSS1995, KSS1998b, PSS2020, RRS2008}  are best possible apart from the $o(n)$ term.
For this reason, also  
\cref{theorem:main-powcycles} and \cref{theorem:main-treewidth} are asymptotically tight for some choices of $k$ and~$r$, e.g.~when $r=k$ or $k = 2$.
We do not know if our 
results are best possible in general.

We remark that  lower bounds for the codegree  in our results can  be inferred from bounds on {\it clique tiling thresholds}.
A \emph{$\clique{k}{r}$-clique tiling} is a collection of vertex-disjoint cliques of order~$r$ in a $k$-graph.
Note that $C^{r-k+1}_{k,n}$ contains a $\clique{k}{r}$-clique tiling if $r$ divides $n$,
and the tree-width of a $\clique{k}{r}$-clique is less than $r$.
This shows that, apart from the $o(n)$ term, the codegree condition in \cref{theorem:main-powcycles,theorem:main-treewidth} could not be lower than the  
codegree threshold $t_{k-1}(\clique{k}{r}, n)$ for finding $\clique{k}{r}$-clique tilings 
in $n$-vertex $k$-graphs (precise definitions are in Section \ref{section:cliquefactors}).

From this connection, our results give upper bounds on $t(\clique{k}{r}, n)$,
in particular showing that $t_{k-1}(\clique{k}{r}) \leq 1 - \left( \binom{r-1}{k-1} + \binom{r-2}{k-2}\right)^{-1}$ holds for all $r \geq k \geq 3$.
For some values of $k, r$ this bound seems to be new, but we emphasize that the better bound $t_{k-1}(\clique{k}{r}, n) \leq 1 - \binom{r-1}{k-1}^{-1}$ was shown by Lo and Markström~\cite{LM2015} to hold for some values of $r, k$, in particular for all $r$ sufficiently large in terms of $k$.
We discuss tiling thresholds with detail in Section \ref{section:cliquefactors}.

Lower bounds on $t(\clique{k}{r}, n)$ can be used as a benchmark for our results, but little is known about these thresholds in general.
Asymptotics~\cite{P2008, LM2015} for $t(\clique{3}{4}, n)$ (and exact values for large $n$~\cite{KM2015})  imply that the term $ 4/5$ in the result from~\cite{BR2020}, and in our  \cref{theorem:main-powcycles} for $(r,k) = (4,3)$ could at best be improved to $ 3/4$.
For $k$ fixed and $r$ large, the main constant in our bounds in \cref{theorem:main-powcycles,theorem:main-treewidth} is of the form $1 - cr^{1-k}$, 	while the clique tiling threshold can be bounded from below by the 
{\it codegree Tur\'an threshold} guaranteeing a single copy of~$\clique{k}{r}$ which in turn has been bounded from below by $1 - c'r^{1-k}\log r$~\cite{LZ2018}, where $c$ and $c'$ are constants (depending on $k$ only).

\subsection{Bounding only the treewidth}
Our result in \cref{theorem:main-treewidth} applies to $k$-graphs $G$ of bounded tree-width with a `hypertree decomposition' of bounded-degree, or equivalently, with the property that every vertex is in a bounded number of bags.
This latter condition is necessary for our approach and it does not follow automatically for any $G$ of bounded maximum degree: there are bounded-degree graphs which do not admit any tree decomposition with bags of optimal size where every vertex is in a bounded number of bags (see, e.g.~\cite[Figure 1]{DO1995}).
However, if we increase the bag sizes  we can obtain a tree decomposition where every vertex is in a bounded number of bags.
We use this fact  to prove a variant of  \cref{theorem:main-treewidth} which applies to arbitrary $k$-graphs of bounded degree and bounded tree-width (see Section~\ref{sec:twvariant}).

\section{Proof overview} \label{section:proofoverview}
Our first main result, \cref{theorem:main-powcycles}, will be deduced as a relatively quick consequence of \cref{theorem:main-treewidth}, so we focus on the proof of the latter result.

Let $G$ be a $k$-graph on $n$ vertices with $\treewidth(G)<r$ that admits an $r$-tree $T$ of bounded vertex-degree, and let $H$ be a $k$-graph satisfying the conditions of \cref{theorem:main-treewidth}.
Instead of working directly with $H$, we will work in the auxiliary \emph{$r$-clique $r$-graph} $K_r(H)$ which captures the $r$-cliques in $H$ (see Section~\ref{def:clique-graph} for a definition).
This is because in order to embed $G$ into $H$, it will be enough to embed $T$ into $K_r(H)$.
We will use a similar strategy as the one used in~\cite{PSS2020}. The difference is that in~\cite{PSS2020}, a codegree condition on a $k$-graph was used to  embed $k$-trees,  while now, wishing to embed an $r$-tree, we do not necessarily have such a condition in the $r$-graph $K_r(H)$.

However, using the minimum codegree condition on the $k$-graph $H$,
we will ensure that $K_r(H)$ satisfies the following three properties, which will be crucial to the proof.

\begin{enumerate}
	\item \label{item:sketch-connections}
	\emph{Connections.}
	Every pair of distinct $(r-1)$-tuples $f,f'\in \partial_{r-1}(K_{r}(H))$ is connected by many tight walks of bounded length.
	
	\item \label{item:sketch-matching}
	\emph{Uniformly dense clique matching.}
	There are a small set $V_0 \subseteq V(H)$ and a partition $\{ V_{j}^i \}_{(i,j)\in [t]\times [r]}$ of $V(H) \setminus V_0$  such for every $i\in [t]$, the $r$-partite graph $K_r(H)[V_{1}^i,\dotsc,V_{r}^i]$ is uniformly dense.
	
	\item \label{item:sketch-absorbers}
	\emph{Absorbers.}
	There is a set $\mathcal A$ of tuples of vertices of $H$ such that for each  $v\in V(H)$, many tuples $\mathbf{v}\in \mathcal A$ can `absorb' $v$.
\end{enumerate}
We embed the absorbers from~\ref{item:sketch-absorbers}, and then
partition $T$ into  a bounded number of small subtrees $T_0, T_1, \dotsc, T_{\ell}$ such that $T_0\cup T_1\cup \dotsb\cup T_i$ is (tightly) connected for each $i\in [t]$. We iteratively embed the trees $T_i$, first finding an element of the  uniformly dense clique matching from~~\ref{item:sketch-matching} having enough unused vertices, and then 
using Property~\ref{item:sketch-connections} to connect the  embedding of $T_0\cup T_1\cup \dotsb\cup T_{i-1}$ with this free space (by embedding the first levels of $T_i$). We then finish by using the absorber, which allow us to increase the embedding of the tree vertex by vertex. 

To carry out the proof as described, we invoke lemmas already present in~\cite{PSS2020}, which allows us to shorten the proof.
The new ingredients here are the connections and the clique matching, whose proofs we describe in more detail now.

We show in Section~\ref{section:connections} that $K_r(H)$ has Property~\ref{item:sketch-connections} via induction on $r$ starting from $k$.
The base case $r = k$ is equivalent to finding bounded-length tight walks between pairs of $(k-1)$-tuples in $k$-graphs with large codegree, and this was already proven in~\cite{PSS2020} (a similar lemma, although with different quantification, was also shown already by R\"odl, Ruci\'nski and Szemer\'edi~\cite[Lemma 2.4]{RRS2008}).
For the induction step, we show here that for $k\le j\le r-1$ and a walk $P$ in $K_j(H)$ of bounded length, the codegree in $H$ allow us to find a bounded number of vertices which together with $P$ define a tight walk in $K_{j+1}(H)$.

To get Property~\ref{item:sketch-matching}, in Section~\ref{section:uniformlydense}, we prove that $K_r(H)$ contains a uniformly dense clique matching using the hypergraph regularity method.
After applying the hypergraph regularity lemma to $H$, it suffices to find an almost spanning $r$-clique matching in the corresponding reduced $k$-graph~$R$.
Unfortunately, the reduced $k$-graph usually does not inherit the codegree conditions of $H$ due to the presence of `irregular' tuples.
We overcome this issue by working in a supergraph $R' \supseteq R$ with the `correct' codegrees, then edge-colouring $R'$ suitably and finding an $r$-clique matching under some colour constraints.
By choosing the edge-colouring appropriately we can make sure that most of the cliques of the found matching are actually contained in $R$.
We combine this with powerful tools to find matchings in hypergraphs by Keevash and Mycroft~\cite{KM2015}.
\section{Tools}
We introduce some notation and results that will be used throughout the paper.
During the paper we will use the common hierarchy notation for quantifiers: the phrase ``$a \ll b$'' in a statement will mean ``for any $b > 0$, there exists $a_0 > 0$ such that for any $0<a \le a_0$, the following holds''.
Sequences of longer chained hierarchies are defined similarly.
If $1/x$ appears in such a hierarchy, we also assume that $x$ is an integer.

\subsection{Ordered shadow}
Given a $k$-graph $H$, its \emph{ordered shadow} $\orderedshadow(H) \subseteq V(H)^{k-1}$ is the set of all tuples of $k-1$ distinct vertices whose (unordered) underlying set belongs to $\partial_{k-1}(H)$. Throughout the paper, we will use bold letters to denote elements from $\partial^\circ(H)$ and for ordered tuples in general. Elements from $\partial^\circ(H)$ will be used as their underlying set for set-theoretical operations. For example, if $\mathbf x=(x_1,\dotsc,x_k),\mathbf y=(y_1,\dotsc,y_k)\in \partial^\circ(H)$, then $\mathbf x\cup\mathbf y=\{x_1,\dotsc,x_k\}\cup\{y_1,\dotsc,y_k\}$.   
\subsection{A function} 
To simplify  notation throughout the paper, we define 
\[f_k(j):=\frac{1}{\binom{j-1}{k-1}+\binom{j-2}{k-2}}\]
for all integers $j\ge k$. Let us observe that $f_k$ is monotone decreasing and $f_k(k)=1/2$.

\subsection{Concentration inequalities} 
During our proofs, it will be useful to have the following standard concentration inequalities at hand (see~\cite[Corollary 2.3 and Corollary 2.4]{JLR2000}).
\begin{lemma} \label{lemma:concentration}
	Let $X$ be a random variable which is the sum of independent $\{0,1\}$-valued random variables. 
	\begin{enumerate}
		\item \label{lemma:happychernoff} If $x \geq 7 \expectation[X]$, then $\probability[X \geq x] \leq \exp(-x)$, and
		\item \label{lemma:sadchernoff} for $0<\eps\le 3/2$, we have
		\[\probability[|X-\expectation[X]|\ge \eps\expectation[X] ]\le \exp(-\eps^2\expectation[X]/3).\]
	\end{enumerate}
\end{lemma}

\subsection{Hypertrees} By  definition, an $r$-tree on $n$ vertices is a hypergraph $T$ with $n - r + 1\ge 1$ edges such that there are orderings $e_1, \dotsc, e_{n-r+1}$ of its edges and $v_{1}, \dotsc, v_{n}$ of its vertices such that $e_1 = \{ v_{1}, \dotsc, v_{r} \}$, and for all $i \in \{r+1, \dotsc, n \}$,
\begin{enumerate}[\upshape (T1)]
	\item \label{hypertree:leafbyleaf} $\{v_{i}\} = e_{i-r+1}\setminus \bigcup_{1\le j < i-r+1} e_j$, and 
	\item \label{hypertree:joinedcorrectly} there exists $j \in [i-1]$ such that $e_{i-r+1} \setminus \{ v_{i}  \} \subseteq e_j$. 
\end{enumerate}
It is clear that an ordering of the edges implies an  ordering of the vertices and vice versa. Any such ordering will be called a \emph{valid ordering}. If~$j\in[i-1]$ is the smallest index such that \ref{hypertree:joinedcorrectly} holds for $e_i$ and $v_{i}$, then we call $e_j$  the \emph{parent} of $e_{i-r+1}$ and $e_{i-r+1}$ a \emph{child} of~$e_j$. 

An \emph{$r$-subtree} of $T$ is an $r$-tree $T'$ such that $T' \subseteq T$. For instance, $e_1, \dotsc, e_{j}$ induces a  $r$-subtree of~$T$, for any $1\le j\le n-r+1$. A \emph{rooted $r$-tree} is a pair $(T,\mathbf x)$ where $T$ is an $r$-tree and $\mathbf x\in \partial^\circ(T)$.

\begin{definition}[Layering]\label{def:layering}
	A \emph{layering} for a rooted $r$-tree $(T,\mathbf x)$, with $\mathbf{x} = (x_1, \dotsc, x_{r-1})$, is a tuple $\mathcal L=(L_1,\dotsc,L_m)$, with $m\in\mathbb N$, such that $\{L_1, \dotsc, L_{m} \}$ partitions $V(T)$, and
	\begin{enumerate}[(L1)]
		\item {$x_i \in L_i$ for all $i \in [r-1]$,}
		and $|L_1| = 1$, \label{item:flatten-root}
		\label{item:flatten-initial}
		\item for each  $v \in L_{i+1}$ with $1\le i< m$ there are $w \in L_{i}$, $e \in E(T)$  such that $\{v, w\} \subseteq e$, 
		\label{item:flatten-degree}
		\item for each $e \in E(T)$, there is $j\in[m]$ such that $|e \cap L_{i}| = 1$ for each $j\le i <j+r$.
		\label{item:flatten-correct}
	\end{enumerate}
	We call the tuple $(T,\mathbf x,\mathcal{L})$ a \emph{layered} $r$-tree.
\end{definition}
We say that $(T,\mathbf x)$ has a \emph{trivial layering} $\mathcal L=(L_1,\dotsc,L_m)$ if $|L_i|=1$  for each $i\in [m]$. Note that a layering $(L_1, \dotsc, L_m)$ of $(T,\mathbf x)$ yields an homomorphism mapping all of~$L_i$ to the $i$th vertex of a tight path.
As an example, Figure~\ref{figure:layering} shows a rooted $3$-tree together with a layering.
All rooted $r$-trees have layerings:

\begin{lemma}[{\cite[Lemma 5.4]{PSS2020}}]\label{lemma:flattening}
	Every rooted $r$-tree $(T,\mathbf {x})$ has a layering.
\end{lemma}

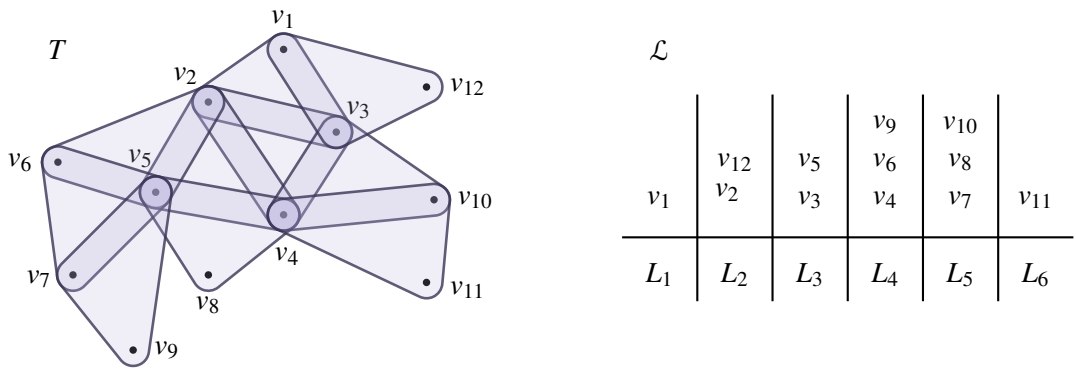
\begin{figure}[h!]
	\centering
	\begin{tikzpicture}[thick, scale=1]
		
		\draw (0,5.4) node {$v_1$};
		\draw (-1.3,4.65) node {$v_2$};
		\draw (1,4.2) node {$v_3$};
		\draw (0.05,2.25) node {$v_4$};
		\draw (-1.9,3.55) node {$v_5$};
		\draw (-3.5,3.5) node {$v_6$};
		\draw (-3.25,1.95) node {$v_7$};
		\draw (-1,1.6) node {$v_8$};
		\draw (-1.55,1) node {$v_9$};
		\draw (2.55,3) node {$v_{10}$};
		\draw (2.45,1.8) node {$v_{11}$};
		\draw (2.45,4.5) node {$v_{12}$};
		
		\draw (-3,5) node {$T$};
		\draw (5,5) node {$\mathcal L$};
		
		\draw (5,3) node {$v_1$};
		\draw (5.9,3.1) node {$v_2$};
		\draw (7,3) node {$v_3$};
		\draw (8,3) node {$v_4$};
		\draw (7,3.5) node {$v_5$};
		\draw (8,3.5) node {$v_6$};
		\draw (9,3) node {$v_7$};
		\draw (9,3.5) node {$v_8$};
		\draw (8,4) node {$v_9$};
		\draw (9,4) node {$v_{10}$};
		\draw (10,3) node {$v_{11}$};
		\draw (6,3.5) node {$v_{12}$};
		
		\draw (5,2) node {$L_1$};
		\draw (6,2) node {$L_2$};
		\draw (7,2) node {$L_3$};
		\draw (8,2) node {$L_4$};
		\draw (9,2) node {$L_5$};
		\draw (10,2) node {$L_6$};
		
		\draw (4.5,2.5) -- (10.5,2.5);
		\draw (5.5,1.65) -- (5.5,4.4);
		\draw (6.5,1.65) -- (6.5,4.4);
		\draw (7.5,1.65) -- (7.5,4.4);
		\draw (8.5,1.65) -- (8.5,4.4);
		\draw (9.5,1.65) -- (9.5,4.4);

		
		\coordinate (v1) at (0,5);
		\coordinate (v2) at (-1,4.3);
		\coordinate (v3) at (0.7,3.9);
		\coordinate (v4) at (0,2.8);
		\coordinate (v5) at (-1.7,3.1);
		\coordinate (v6) at (-3,3.5);
		\coordinate (v7) at (-2.8,2);
		\coordinate (v8) at (-1,2);
		\coordinate (v9) at (-2,1);
		\coordinate (v10) at (2,3);
		\coordinate (v11) at (1.9,1.9);
		\coordinate (v12) at (1.9,4.5);	
		
		\tikzstyle{every node}=[circle, draw, fill, inner sep=0pt, minimum width=2pt]
		\draw (v1) node {};
		\draw (v2) node {};
		\draw (v3) node {};
		\draw (v4) node {};
		\draw (v5) node {};
		\draw (v6) node {};
		\draw (v7) node {};
		\draw (v8) node {};
		\draw (v9) node {};
		\draw (v10) node {};
		\draw (v11) node {};
		\draw (v12) node {};
		
		\triple{(v2)}{(v1)}{(v3)}{6pt}{1pt}{DarkDesaturatedBlue,opacity=0.8}{VerySoftBlue,opacity=0.2};
		\triple{(v2)}{(v3)}{(v4)}{6pt}{1pt}{DarkDesaturatedBlue,opacity=0.8}{VerySoftBlue,opacity=0.2};
		\triple{(v2)}{(v4)}{(v5)}{6pt}{1pt}{DarkDesaturatedBlue,opacity=0.8}{VerySoftBlue,opacity=0.2};
		\triple{(v2)}{(v5)}{(v6)}{6pt}{1pt}{DarkDesaturatedBlue,opacity=0.8}{VerySoftBlue,opacity=0.2};
		\triple{(v6)}{(v5)}{(v7)}{6pt}{1pt}{DarkDesaturatedBlue,opacity=0.8}{VerySoftBlue,opacity=0.2};
		\triple{(v5)}{(v4)}{(v8)}{6pt}{1pt}{DarkDesaturatedBlue,opacity=0.8}{VerySoftBlue,opacity=0.2};
		\triple{(v7)}{(v5)}{(v9)}{6pt}{1pt}{DarkDesaturatedBlue,opacity=0.8}{VerySoftBlue,opacity=0.2};
		\triple{(v4)}{(v3)}{(v10)}{6pt}{1pt}{DarkDesaturatedBlue,opacity=0.8}{VerySoftBlue,opacity=0.2};
		\triple{(v4)}{(v10)}{(v11)}{6pt}{1pt}{DarkDesaturatedBlue,opacity=0.8}{VerySoftBlue,opacity=0.2};
		\triple{(v3)}{(v1)}{(v12)}{6pt}{1pt}{DarkDesaturatedBlue,opacity=0.8}{VerySoftBlue,opacity=0.2};
		
	\end{tikzpicture}
	\label{figure:layering}
	\caption{On the left, a $3$-tree $T$ with valid ordering $v_1, \dotsc, v_{12}$.
		On the right, a table shows the layering $\mathcal L=(L_1,\dotsc, L_6)$ of $T$ when it is rooted at $\mathbf x=(v_1,v_2)$.} 
\end{figure}

Given a layered $r$-tree $(T,\mathbf x,\mathcal{L})$, with $\mathcal{L} = (L_1, \dotsc, L_m)$,
and given $\mathbf {s}\in\orderedshadow(T)$, we say $\mathbf {s}=(s_1,\dotsc,s_{r-1})$ is \emph{$\mathcal{L}$-layered} if {for some $j \in [m]$, we have $s_i \in L_{i+j-1}$ for each $i \in [r-1]$.
	That is, $\mathbf s$ meets $r-1$ consecutive layers of $\mathcal{L}$.}
In that case we say that $j$ is the \emph{rank} of $\mathbf s$. Furthermore, the tree  \textit{induced by $\mathbf s$} is the $r$-subtree $T_{\mathbf s}$ of $T$ spanned by $\bigcup_{i\ge 0}E_i$, where $E_0$ is the child of $\mathbf s$ and $E_{i+1}$ contains all children of edges in $E_i$. The induced subtree $(T_{\mathbf s},\mathbf s)$ has a natural layering $\mathcal L^{\mathbf s}$, called \emph{inherited layering}, which consists of  the sets $L_j\cap V(T_{\mathbf s})$. We write $T-T_{\mathbf s}$ for the tree obtained from $T$ by deleting all edges in $E(T_{\mathbf s})$, and all vertices in $V(T_{\mathbf s})\setminus {\mathbf s}$.

\begin{lemma}[{\cite[Lemma 5.12]{PSS2020}}]\label{lemma:cutting1}
	Let $\Delta\ge r\ge 2$, and let $(T,\mathbf {x},\mathcal L)$ be a layered $r$-tree with  $\mathcal L=(L_1,\dotsc,L_{m})$ and $\Delta_1(T) \leq \Delta$. Let $F\subseteq E(T)$ be the set of all edges meeting $L_1$ and {let $\mathbf {S} \subseteq \orderedshadow(T)$ consist of all the $\mathcal{L}$-layered tuples whose unordered vertices are in $\{e\setminus L_1: e\in F\}$.} Then,
	\begin{enumerate}[\upshape(i)]
		\item each $\mathbf s\in \mathbf S$ is $\mathcal L$-layered and has rank $2$,
		\item $|F|=|\mathbf S|\le\Delta$, and
		\item $E(T)=F\sqcup\bigsqcup_{\mathbf {s} \in \mathbf{S}}E(T_{\mathbf s})$.
	\end{enumerate}
\end{lemma}
With this at hand, we can find a small subtree of a given layered $r$-tree.

\begin{lemma} \label{lemma:smallsubtree}
	Let $\Delta\ge r\ge 2$ and let $1/n\ll \gamma,\Delta$. Let $(T,\mathbf x,\mathcal L)$ be a layered $r$-tree on $n$ vertices with $\Delta_1(T)\le \Delta$. Then there exists $\mathbf x'\in\orderedshadow(T)$ such that
	\begin{enumerate}
		\item $\mathbf x'$ is $\mathcal L$-layered,
		\item $V(T')\cap V(T-T_{\mathbf{x}'})=\mathbf x'$, and
		\item ${\gamma n}/(2\Delta)\le|E(T_{\mathbf{x}'})|\le \gamma n$.
	\end{enumerate} 
\end{lemma}
\begin{proof}
	Let $\mathbf x'$ be an $\mathcal L$-layered tuple with highest rank such that $E(T_{\mathbf{x}'})\ge \gamma n/(2\Delta)$. Because of Lemma~\ref{lemma:cutting1} for $(T_{\mathbf{x}'},\mathbf x',\mathcal L^{\mathbf x'})$, there exists a set $F\subseteq E(T_{\mathbf{x}'})$ and $\mathbf S\subseteq \orderedshadow(T_{\mathbf{x}'})$ with $|F|=|\mathbf S|\le \Delta$ such that $E(T_{\mathbf{x}'})=F\sqcup\bigsqcup_{\mathbf s\in\mathbf S}E(E_{\mathbf s})$. Using the maximality of $\mathbf x'$ and $1/n\ll \Delta,\gamma$, we have
	\begin{align*}|E(T_{\mathbf{x}'})| & \le \Delta +\Delta \cdot\frac{\gamma n}{2\Delta}\le \gamma n. \qedhere \end{align*}
\end{proof}
\subsection{The clique $r$-graph}\label{def:clique-graph}
Let $r\ge k\ge 2$.
Given a $k$-graph $H$,
the \emph{$r$-clique $r$-graph} $K_r(H)$ is the $r$-graph on $V(H)$ whose edges correspond to $r$-cliques of $H$.
In particular, $K_k(H)=H$ and $\partial_k(K_r(H)) \subseteq H$. Similarly, $J \subseteq K_r(\partial_k(J))$ for any $r$-graph $J$.

Given $F\in K_{r-1}(H)$,  write $\deg_{r}(F)= \deg_{K_r(H)}(F)$ for the number of $r$-cliques in $H$ that contain $F$. We can bound $\deg_r(F)$ by discarding from $V(H)-F$ all vertices which do not form an edge in $H$ with at least one of the $(k-1)$-subsets of $F$, this gives
\begin{align}\label{eq:r-degree}
	\deg_r(F)\ge n-(r-1)-\sum_{f\in \binom{V(F)}{k-1}}(n-(r-1)-\deg_H(f)).
\end{align}
\begin{lemma} \label{lemma:crucialbound}
	Let $j\ge k \geq 2$,  $\alpha>0$, and let $H$ be a $k$-graph on $n\ge (j+1)^2/\alpha$ vertices with $\delta(H) \geq \left(1-f_k(j+1) + \alpha \right)n$. Then for any  $C_1, C_2 \in K_{j}(H)$ with $|C_1 \cap C_2| = j-1$,
	there are at least $\alpha n$ vertices $v$ such that $C_1 \cup \{v\}, C_2 \cup \{v\}\in K_{j+1}(H)$.
\end{lemma}
\begin{proof}
	We count the number of $(k-1)$-sets completely contained in either $C_1$ or $C_2$.
	Each of $C_1, C_2$ has $\binom{j}{k-1}$ such sets of which $\binom{j-1}{k-1}$ are shared,
	so by inclusion-exclusion and by Pascal's identity, the desired number is $2\binom{j}{k-1} - \binom{j-1}{k-1} = \binom{j}{k-1}+\binom{j-1}{k-2}$.
	So, the number of possible choices for $v$ is at least 
	\begin{align*}
		n-j-1-\left(\tbinom{j}{k-1}+\tbinom{j-1}{k-2}\right)\left (n-\left(1-f_k(j+1)+\alpha\right)n\right)& \ge \alpha n. \qedhere 
	\end{align*}
\end{proof}
\begin{corollary} \label{corollary:extendtwovertices}
	Let $r \ge k \geq 2$, $\alpha>0$ and let $H$ be a $k$-graph on $n\ge r^2/\alpha$ vertices with $\delta(H) \geq \left(1 - f_k(r) + \alpha \right)n$.   Then for any distinct $u, v\in V(H)$ there are at least $(\alpha n)^{r-1}$  cliques $C \in K_{r-1}(H)$ such that $C \cup \{u\},C \cup \{v\}\in K_r(H)$.
\end{corollary}
\begin{proof}
	For $k-1\le j\le r-1$,  let $Q_j$ be the number of $j$-sets of vertices $X$ such that both $X\cup\{u\}$ and $X\cup\{v\}$ belong to $K_{j+1}(H)$.
	We will prove inductively that $Q_j\ge \alpha^jn^j$.
	For $j=k-1$, note that since $ \delta(H) \ge n/2+\alpha n$, for any $(j-1)$-set $S\subseteq V(H) \setminus \{u,v\}$,
	the sets $S\cup\{u\}$ and $S\cup\{v\}$ have at least $2\alpha n$ common neighbours.
	This yields $Q_{k-1}\ge\alpha^{k-1} n^{k-1}$.
	
	For the inductive step, let $k-1 \leq j < r-1$ 	and consider any $X\in K_{j}(H)$ with $X\cup \{u\}, X\cup \{v\}\in K_{j}(H)$.
	Use Lemma~\ref{lemma:crucialbound}, with $C_1=X\cup \{u\}$ and $C_2=X\cup \{v\}$, to see that there are at least $\alpha n$ many vertices $w$ such that both $X\cup \{u,w\}$ and $X\cup \{v,w\}$ belong to $K_{j+1}(H)$.
	So $Q_{j+1}\ge \alpha nQ_{j}\ge \alpha^{j}n^{j}$ (the second inequality holds by induction).
\end{proof}

\subsection{Reservoir}

Given a $k$-graph $H$, a $(k-1)$-set $f \in \partial_{k-1}(H)$,
and a subset of vertices $W\subseteq V(H)$,
we define $\deg_H(f, W) = \left|\{v\in W: f\cup\{v\}\in H \}\right|$.
Given an edge $e\in H$, let $d_H^K(e)$ be the number of copies of $K^{(k)}(2)$ in $H$ in which $e$ participates, where $K^{(k)}(2)$ is the complete $k$-partite $k$-graph with parts of size $2$.
\begin{definition}
	Let $\gamma, \mu>0$ and let $H$ be a $k$-graph on $n$ vertices. We say that $W\subseteq V(H)$ is a {$(\gamma,\mu)$-reservoir} for $H$ if
	\begin{enumerate}
		\item $|W|=(\gamma\pm \mu)n$,
		\item {for every $f \in \partial_{k-1}(H)$ we have $\deg_H(f, W)\ge (\deg_H(f)/n- \mu)|W|$, and}
		\item for every $e\in H$, we have $d_H^K(e,W)\ge (d_H^K(e)/\binom {n-k}k-\mu)\binom{|W|-k}k$.
	\end{enumerate}	
\end{definition}

The following lemma was proved in \cite{PSS2020} using an standard probabilistic argument.

\begin{lemma}[{\cite[Lemma 6.7]{PSS2020}}]\label{lem:reservoir}
	Let $1/n\ll \mu\ll \gamma\le 1$. Then every $r$-graph $H$ on $n$ vertices has a
	$(\gamma,\mu)$-reservoir. 
\end{lemma}

\section{Connections}\label{section:connections}

An ordered sequence $(x_1, \dotsc, x_n)$ of vertices from a $k$-graph $H$ is a \emph{walk} if every $k$ consecutive vertices form an edge.
We  often write a walk as the juxtaposition of its vertices, i.e. $x_1 \dotsb x_n$.
The \emph{length} of a walk is the number of its edges counted with possible repetitions, e.g. $x_1 \dotsb x_n$ has length $n-k+1$. 

For a walk $W = x_1 \dotsb x_n$ in $H$, the \emph{start of $W$} is  $\sta{W}:=(x_1, \dotsc, x_{k-1})\in \orderedshadow(H)$, and the \emph{end of $W$} is $\ter{W}:=(x_{n-k+1}, \dotsc, x_{n})\in\orderedshadow(H)$. We also say that $W$ \emph{goes from $\sta{W}$  to $\ter{W}$}.
The \emph{interior} $\int{W}$ of $W$ is the set $V(W)\setminus(\sta{W}\cup \ter{W})$, and we say that $W$ \emph{has $|\int{W}|$ internal vertices}. 
Note that $|\int{W}| \leq n - 2k - 2$. If $\int{W}\cap S=\emptyset$, we call $W$  \emph{internally disjoint} from~$S$. 

We say that $H$ is \emph{$(\alpha,\ell)$-connecting} if there exists $q\le \ell$ such that for every $\mathbf{v}_1,\mathbf{v}_2\in\orderedshadow (H)$ there are at least $\alpha^qv(H)^q$ walks in $V(H) \setminus (\mathbf{v}_1\cup \mathbf{v}_2)$ which go from $\mathbf{v}_1$ to $\mathbf{v}_2$, each of length $\ell$ and with exactly $q$ internal vertices.

\begin{lemma} \label{lemma:connectivity}
	Let $r\ge k\ge 2$, let  $\ell\ge  4kr!$ and let $1/n\ll\alpha \ll 1/\ell$. If $H$ is a  $k$-graph  on $n$ vertices with 
	$\delta(H) \geq \big(1 - f_k(r) + \alpha \big)n$,   then $K_{r}(H)$ is $({\alpha}/{2},\ell)$-connecting.
\end{lemma}

In order to prove Lemma~\ref{lemma:connectivity}, we will use the following lemma for which we need a quick definition. We say that a $k$-graph $H$ is {\em $(m,U)$-large}, for $U\subseteq V(H)$, if any two $f, f' \in \partial_{k-1}(H)$ have at least $m$ common neighbours in $U$. Note that the $k$-graph $H$ from Lemma~\ref{lemma:connectivity} is $(2\alpha n,V(H))$-large.

\begin{lemma}[{\cite[Lemma~7.1]{PSS2020}}]\label{lemma:strengthenedconnectinglemma}
	For $k, \ell$ with $k \ge 2$ and $\ell \ge (2k+1)\lfloor k /2 \rfloor + 2k$, there exists $q \leq \ell$ with the following property.
	Let $1/n \ll \gamma \ll 1/k, 1/\ell$. Let $H$ be a $k$-graph on $n$ vertices which is $(\gamma n, U)$-large for some $U \subseteq V(H)$,
	and let $\mathbf{v}_1, \mathbf{v}_2 \in \orderedshadow(H)$.
	Then there are at least $(\gamma n/2)^{q}$ walks $W$,
	each of length $\ell$ and with $q$ internal vertices,
	which go from $\mathbf{v}_1$ to $\mathbf{v}_2$ and are internally disjoint from $(V(H) \setminus U) \cup \mathbf{v}_1 \cup \mathbf{v}_2$.
\end{lemma}
\begin{proof}[Proof of Lemma~\ref{lemma:connectivity}]
	We consider a slightly stronger definition of connectivity for this proof.
	Given $\alpha>0$ and $\ell,q, x\in\mathbb N$, we say that $H$ is \emph{$(\alpha,\ell,q, x)$-connecting} if for every $\mathbf{v}_1,\mathbf{v}_2\in\orderedshadow (H)$ and every set $X\subseteq V(H)$ with $|X|\le x$, there are at least $\alpha^qv(H)^q$  walks which go from $\mathbf{v}_1$ to $\mathbf{v}_2$, each of length $\ell$, with $q$ internal vertices, and avoiding $\mathbf{v}_1\cup \mathbf{v}_2\cup X$.
	
	In order to prove the lemma, we will prove the following:
	\begin{align}
		\parbox[c]{0.8\linewidth}{for all $i$ with $k\le i \le r$,
			and for all $\ell \geq t \geq 4 k i!$,
			there exists $q \leq t$ such that the $i$-graph $K_{i}(H)$ is $({\alpha}/{2}, t, q, 2\ell(r-i) )$-connecting.}
		\label{equation:connectinginductive}
	\end{align}
	{We prove this by induction on $i$.
		Note that since $1/n \ll \alpha \ll 1/\ell, 1/k, 1/r$,
		the removal of any $2\ell(r-k)$ vertices $X$ from a $(2 \alpha n, V(H))$-large graph results in a $(\alpha n, V(H) \setminus X)$-large graph.
		Thus Lemma~\ref{lemma:strengthenedconnectinglemma}  proves the base case $i=k$.}
	
	Now suppose that $k < i \leq r$.
	Let $\ell_{i-1} = 4k(i-1)!$.
	By the induction hypothesis, there exists $q_{i-1} \leq \ell_{i-1}$ such that $K_{i-1}(H)$ is $(\alpha/2, \ell_{i-1}, q_{i-1}, 2\ell(r-i+1))$-connecting. Let $\ell_i = i(\ell_{i-1}-i+1)$ and $q_i:=q_{i-1}+\ell_{i-1}-i+1$.
	We will start by showing that $K_{i}(H)$ is $(\alpha/2,\ell_i, q_i, 2\ell(r-i)+\ell)$-connecting. 
	
	Let $\mathbf{v}_1, \mathbf{v}_2 \in \partial^\circ(K_i(H))$ and $X \subseteq V(K_i(H))$ with $|X| \leq 2\ell(r-i)+\ell$ be arbitrary. Now let $\mathbf{v}_1'=(v_2,\dotsc,v_{i-1})$ and $\mathbf{v}_2'=(u_1,\dotsc,u_{i-2})$ 
	and $X':=X \cup\{v_1, u_{i-1}\}$, so $|X'| \leq |X|+2 \leq 2\ell(r-i)+\ell+2 \leq 2\ell(r-i+1)$.
	Since $K_{i-1}(H)$ is $(\alpha/2, \ell_{i-1}, q_{i-1}, 2\ell(r-i+1))$-connecting, there are at least $(\alpha/2)^{q_{i-1}}n^{q_{i-1}}$ walks $W=(w_1, w_2, \dotsc ,w_{\ell_{i-1}+k-1})$ in $K_{i-1}(H)$ which go from  $\mathbf{v}_1'$ and $\mathbf{v}_2'$ whose interior avoids $\mathbf{v}_1'\cup \mathbf{v}_2'\cup X'=\mathbf{v}_1\cup \mathbf{v}_2\cup X$.
	That is, $\sta{W}=(v_2,\dotsc,v_{i-1})$, $\ter{W}=(u_1,\dotsc,u_{i-2})$, and every ${i-1}$ consecutive vertices in $W$ belong to $K_{i-1}(H)$.
	Also, each $W$ uses exactly $q_{i-1}$ internal vertices.
	
	For each such $W$, and for each $1\le j \le \ell_{i-1}-i+2$, we apply Lemma~\ref{lemma:crucialbound} to $H-W$ to find at least  $\alpha n/2$ vertices $x_j \in V(H)\setminus (V(W)\cup X)$ such that
	\begin{align}
		\{x_j,w_j,w_{j+1},\dotsc,w_{j+i-2}\}, \{x_j,w_{j+1},\dotsc,w_{j+i-1}\}\in K_{i}(H).
	\end{align}
	Then the walk $W'$ given by
	$$w_1w_2\dotsb w_{i} x_1 w_2 w_3 \dotsb w_{i+1} x_2 w_3 w_4 \dotsb w_{i+2} x_3 \dotsb x_{\ell_{i-1}-i+1} w_{\ell_{i-1}-i+2} w_{\ell_{i-1}-i+4}\dotsb w_{\ell_{i-1}}$$
	is a walk
	in $K_{i}(H)$ that goes from $\mathbf{v}_1$ to $\mathbf{v}_2$, has length $i(\ell_{i-1}-i+1) = \ell_i$ and uses $q_i$ internal vertices, and avoids $\mathbf{v}_1\cup \mathbf{v}_2\cup X$.
	As there are at least 
	$(\alpha/2)^{q_{i-1}}n^{q_{i-1}}$ walks $W$ to start with,
	and $\alpha n/2$ choices for each of the vertices~$x_j$, we deduce that there are at least $$(\alpha/2)^{q_{i-1}}n^{q_{i-1}}(\alpha n/2)^{\ell_{i-1}-i+1}= (\alpha/2)^{q_i}n^{q_i}$$ such walks $W'$.
	This proves that $K_{i}(H)$ is $(\alpha/2,\ell_i, q_i, 2\ell(r-i)+\ell)$-connecting.
	
	To finish the proof of \eqref{equation:connectinginductive},
	we need to find walks of any length $t$ satisfying $\ell \geq t > \ell_i$ avoiding any given $X \subseteq V(H)$ of size at most $2 \ell(r-i)$, for that we proceed as follows.
	Starting from any $\mathbf{v}_1$,
	use the degree in $K_i(H)$ to greedily find the first $t-\ell_i$ vertices of a walk starting in $\mathbf{v}_1$, and say those new vertices are $X'$.
	Clearly $X'$ can be chosen in at least $(\alpha n / 2)^{t - \ell_i}$ different ways.
	Next, since $|X \cup X'| \leq 2 \ell(r-i) + (t - \ell_i) \leq 2\ell(r-i)+\ell$,
	we can finish the walks using that $K_{i}(H)$ is $(\alpha/2,\ell_i, q_i, 2\ell(r-i)+\ell)$-connecting to avoid $X \cup X'$.
\end{proof}

We now show how to connect vertex-disjoint ordered sets from the shadow of the $r$-clique graph of our host graph with tight paths (not just walks) in a controlled way. 

\begin{lemma}\label{lemma:connectivity:cycles}
	Let $r\ge k\ge 2$, $\ell\ge  5kr!$ and $1/n\ll\alpha, 1/\ell \le 1$. Let $H$ be an $n$-vertex $k$-graph  with $\delta(H) \geq \left(1 -f_k(r) + \alpha \right)n.$	Then for each $U\subseteq V(H)$ with $|U|\le \alpha n/2$, and for each vertex-disjoint pair $\mathbf{v}_1,\mathbf{v}_2\in \partial^\circ (K_r(H-U))$, there is a tight path $P$ of length $\ell$ in $K_r(H-U)$  with $\sta{P}=\mathbf{v}_1$ and $\ter{P}=\mathbf{v}_2$.
\end{lemma}

For the proof of Lemma~\ref{lemma:connectivity:cycles}, we need some notation. 
Given $\ell_1\le \ell_2\le m$, and a layering $\mathcal L=(L_1, \dotsc, L_{m})$ of a rooted $r$-tree $(T, \mathbf x)$,
set $V_{[\ell_1,\ell_2]}:=\bigcup_{\ell_1\le i\le \ell_2} L_i$. 
If  $|L_i|\le M$ for each $\ell_1\le i\le\ell_2$, we say $V_{[\ell_1,\ell_2]}$ is {\em $M$-bounded}.
We write $\mathbf y\in \partial^\circ (T[V_{[\ell_1,\ell_2]}])$ if there is an index $j\in [\ell_1, \ell_2-r+1]$ such that $\mathbf y=(y_1,\dotsc, y_{r-1})\in \partial^\circ (T)$ with $y_i\in L_{j+i-1}$ for all $i\in [r-1]$.
The following auxiliary result from~\cite{PSS2020} essentially says that whenever we can connect many tuples with many paths of length $\ell$, we can also embed arbitrary $r$-trees which have $\ell$ layers using those tuples.

\begin{lemma}[{\cite[Lemma~8.4]{PSS2020}}] \label{lemma:embeddingtrunk}
	Let $1/n \ll 1/r, 1/q, 1/\ell, 1/M, \delta, \alpha$ and $2\le r \le (\ell-1)/2$.
	Let $\mathcal{L} = (L_1, \dotsc, L_{m})$ be a layering of a rooted $r$-tree $(T, \mathbf x)$, and assume $V_{[t, t+\ell]}$ is $M$-bounded. \\
	If $T_I \subseteq T[V_{[t, t+\ell]}]$, $G$ is an $r$-graph on $n$ vertices, $U \subseteq V(G)$ and $\mathbf F_1, \mathbf F_2 \subseteq \partial^\circ (G)$ are such that
	\begin{enumerate}[\rm (I)]
		\item $|\mathbf F_1|, |\mathbf F_2| \ge \delta n^{r-1}$, and \label{item:embedding-manyends}
		\item for every $\mathbf v_1 \in \mathbf F_1$ and $\mathbf v_2 \in \mathbf F_2$ there are at least $\alpha n^q$ many walks going from $\mathbf v_1$ to $\mathbf v_2$, each of length $\ell-r+1$, each with $q$ internal vertices all from $U\setminus \mathbf v_1\cup \mathbf v_2$, \label{item:embedding-manywalks}
	\end{enumerate}
	then there exists an embedding $\phi: V(T_I) \rightarrow V(G)$ such that
	\begin{enumerate}
		\item $\phi(\mathbf v)\in  \mathbf F_1$ for each $\mathbf v\in \partial^\circ (T_I[V_{[t,t+r-2]}])$,
		\item $\phi(\mathbf v)\in  \mathbf F_2$ for each $\mathbf v\in \partial^\circ (T_I[V_{[t+\ell - r + 2,t+\ell]}])$, and
		\item $\phi(v) \in U$ for each $v \in V_{[t+r-1, t+\ell-r+1]}$.
	\end{enumerate}	
\end{lemma}
\begin{proof}[Proof of Lemma~\ref{lemma:connectivity:cycles}]
	{
		Choose constants $1/n\ll \mu \ll\delta\ll\alpha$.
		Set $H'=H-U$. Because of the minimum codegree in $H$, for every $f\in \partial_{k-1}(H')$ we have $\deg_{H'}(f) \geq (1 - f_k(r) + \alpha/2) n$.
		By using Lemma~\ref{lemma:crucialbound} we conclude that for all $f\in\partial_{r-1}(K_r(H'))$,
		\begin{align}\label{eq:degree:cliques:W}
			\deg_{K_r(H')}(f)\ge \frac{\alpha}{2}|V(H')|.
	\end{align}}
	{Let $\mathbf{v}_1=(v_1,\dotsc, v_{r-1})$.
		Use~\eqref{eq:degree:cliques:W} iteratively $r-1$ times,
		to find greedily a tuple $(u_1, \dotsc, u_{r-1})$
		such that $\mathbf{v}_1 u_1 \dotsb u_{r-1}$ is a tight path on $2r-2$ vertices in $K_r(H')$.
		Since there were at least $\alpha |V(H')|/2 - (r-1) \geq \alpha |V(H')|/3$ choices for each $u_i$,
		there are at least $2 \delta |V(H')|^{r-1}$ possible options for the tuple $(u_1, \dotsc, u_{r-1})$. Let $\mathbf{F}'_1$ be the set of all those tuples and construct $\mathbf{F}'_2$ analogously.}
	
	{Let $H'' = H \setminus (\mathbf{v}_1 \cup \mathbf{v}_2)$ and $\mathbf{F}_1$ to be the set of tuples of $\mathbf{F}'_1$ which survive in $H''$, and
		similarly define $\mathbf{F}_2$.
		Clearly $\deg_{K_r(H'')}(f) \geq \alpha |V(H'')|/3$
		and $|\mathbf{F}_1|, |\mathbf{F}_2| \geq \delta |V(H'')|^{r-1}$.
		Because of Lemma~\ref{lemma:connectivity}, we know that $K_r(H'')$ is $(\alpha/6,\ell-2r+2)$-connecting.
		In particular, there exists $q \leq \ell - 2r + 2$ such that for every choice of $\mathbf{v}'_i \in \mathbf{F}_i$
		there are at least $(\alpha/6)^q V(H')^q$ walks in $V(H'') \setminus (\mathbf{v}'_1 \cup \mathbf{v}'_2)$ which go from $\mathbf{v}'_1$ to $\mathbf{v}'_2$ and have length $\ell - 2r + 2$ and exactly $q$ internal vertices.
		Thus we can use Lemma~\ref{lemma:embeddingtrunk} with $T$ being a tight path of length $\ell - 2r + 2$, which is clearly $1$-bounded with the trivial layering with one vertex in each cluster.
		This gives a tight path of length $\ell - 2r + 2$ in $K_r(H'')$
		which connects some $\mathbf{v}'_1 \in \mathbf{F}_1$ with some $\mathbf{v}'_2 \in \mathbf{F}_2$.
		In turn, this yields a tight path of length $\ell$ which goes from $\mathbf{v}_1$ to $\mathbf{v}_2$ and avoids $U$.}
\end{proof}

\section{Uniformly dense clique matching} \label{section:uniformlydense}

In this section we show a matching lemma which is required in the proof of Theorem~\ref{theorem:main-treewidth}.

\subsection{Matching lemma and outline of its proof}
Given a $k$-graph $G$ and disjoint subsets $V_1,\dotsc,V_k \subseteq V(G)$,
let $G[V_1, \dotsc, V_k]$ be the $k$-partite graph with vertex classes $V_1, \dotsc, V_k$, having all edges of $G$ that intersect each of the sets $V_i$ exactly once.
Let $e_G(V_1, \dotsc, V_k)$ denote the number of edges of $G[V_1, \dotsc, V_k]$.
A $k$-partite $k$-graph $G$ on sets $V_1, \dots, V_k$ of the same size $m$  is called \emph{$(\varepsilon, d)$-uniformly dense} if for all choices of $W_i \subseteq V_i$ with $|W_i| = h \geq \varepsilon m$ for all $1 \leq i \leq k$, 
\begin{equation}
	e_G(W_1, \dotsc, W_k) \geq d h^k.
\end{equation} 
Let $G$ be a $k$-graph, let $\varepsilon,d>0$ and $t\in\mathbb N$.
We say that $G$ has an \emph{$(\varepsilon,d)$-uniformly dense perfect matching} $\mathcal M=\{(V_1^i,\dotsc,V_k^i)\}_{i\in [t]}$ and we call sets $\{V_j^i\}_{j\in [k]}$ its {\em edges} if 
\begin{enumerate}
	\item $\{V_a^i\}_{a\in [k],i\in [t]}$ partitions $V(G)$,
	\item $|V_a^i|=|V_{b}^j|$ for all $a,b\in [k]$, $i,j\in [t]$, and
	\item $G[V_1^i,\dotsc,V_k^i]$ is $(\varepsilon,d)$-uniformly dense for each $i\in [t]$.
\end{enumerate}
The aim of this section is to show the following lemma.

\begin{lemma}[Matching lemma] \label{lemma:uniformlydensematching}
	Let $r\ge k\ge 2$, let $1/n\ll 1/M\ll d \ll \varepsilon, \delta, \alpha,1/r$, and let $H$ be an $n$-vertex $k$-graph with
	$\delta(H) \geq \left(1 -f_k(r) + \alpha \right)n$.	Then $H$ has a subgraph $H'\subseteq H$ with $|V(H')|\ge (1-\delta)n$ such that $K_{r}(H')$ admits an $(\varepsilon,d)$-uniformly dense perfect matching with at most $M$ edges.
\end{lemma}

The proof can be summarised in the following three steps.

\begin{enumerate}
	\item \label{item:matchingsteps-partition}
	\emph{Obtaining a regular partition.}
	An application of hypergraph regularity to $H$ will yield a partition of $H$ into few parts $\mathcal{P}$ and a  reduced $k$-graph $R$ on vertex set $\mathcal{P}$,
	consisting on $k$-tuples of clusters which are both `dense' and `regular'.
	Each clique on $r$ vertices in $R$ will correspond to a collection of $r$ clusters which could form an edge in a uniformly dense cluster matching of $H$.
	Given this, the goal is to find a $\clique{k}{r}$-matching (for definitions see below)
	in $R$ which is as large as possible.
	
	\item \label{item:matchingsteps-factor}
	\emph{Finding a large clique factor.}
	Our approach to find such a matching is to use the codegree tiling threshold for clique factors.
	More precisely, in this step we will verify that the  codegree of $H$ suffices to find a $\clique{k}{r}$-factor in $H$, that is, a $\clique{k}{r}$-matching which covers every vertex.
	Hence, if we could show that $R$ inherits the relative codegree conditions of $H$ (up to scaling and  a small error),
	we would find the desired large $\clique{k}{r}$-matching in $R$ and we would be done.
	
	\item \label{item:matchingsteps-irregular}
	\emph{Avoiding the irregular edges.}
	However, an obstacle is that $R$ usually will \emph{not} inherit the codegree conditions of $H$ due to the presence of `irregular' edges, which constitute a $k$-graph $I$ on $V(R)$.	Since the union graph $R \cup I$ \emph{will} inherit the codegree conditions of $H$, we can use tools from the previous step to find a clique factor in $R \cup I$.
	We would like to have the factor  avoiding  as much as possible of $I$, which can be done by defining an edge-colouring  and introducing some restraints with respect to this edge-colouring.
\end{enumerate}

Although the structure of the proof is simple,  the use of  hypergraph regularity requires  long and technical definitions, whose use has however become quite standard.
To avoid clutter here, we defer the technical details of step \ref{item:matchingsteps-partition}
to Appendix~\ref{appendix:regularity},
where the proof of \cref{lemma:uniformlydensematching} is given in its entirety.
In what remains of this section we will prove step \ref{item:matchingsteps-factor} (in Lemma~\ref{lemma:cliquetilingbounds}) and step \ref{item:matchingsteps-irregular} (in \cref{lemma:Hfactoravoiding}), which constitute the novel and non-standard ingredients of the proof.

We also remark that (as pointed to us by a referee) the outcome of steps \ref{item:matchingsteps-factor}--\ref{item:matchingsteps-irregular} can be achieved in an alternative way by considering the codegree threshold of \emph{almost} perfect clique tilings.
We decided to keep our original proof because we believe it to be novel and independently interesting, but we will sketch this alternative briefly in \Cref{remark:almost}.

\subsection{Clique factors in hypergraphs} \label{section:cliquefactors}
For $k$-graphs $G$ and $H$,
an \emph{$H$-matching} in $G$ is a set of vertex-disjoint copies of $H$ in $G$, and  an \emph{$H$-factor} is a spanning $H$-matching.
Clearly, if $G$ has an $H$-factor, then $|V(H)|$ divides $|V(G)|$.
Given $1 \leq d < k$ and
an $r$-vertex $k$-graph $H$,  
the \emph{$d$-degree $H$-tiling threshold} $t_d(H,n)$ is the minimum $h$ such that every $n$-vertex $k$-graph~$G$  with $\delta_d(G) \geq h$ and $n$ divisible by $r$ has an $H$-factor.
The \emph{asymptotic $d$-degree $H$-tiling threshold} is $t_d(H) = \limsup_{n \rightarrow \infty} t_d(H,nr)/\binom{nr-d}{k-d}$. 

For $r \geq k \geq 2$, $\clique{k}{r}$ is the $k$-uniform clique on $r$ vertices.
The Hajnal-Szemer\'edi theorem~\cite{HS1970} states that $t_1(\clique{2}{r}) = 1 - 1/r$.
Estimates on $t_{k-1}(\clique{k}{r})$ for $r \geq k \geq 3$ are given in the following lemma, which is the core ingredient to find the clique factor needed in step \ref{item:matchingsteps-factor} of our proof of \cref{lemma:uniformlydensematching}.

\begin{lemma} \label{lemma:cliquetilingbounds}
	For all $r \geq k \geq 3$, we have  $ t_{k-1}(\clique{k}{r}) \leq 1 - f_k(r).$
\end{lemma}

We point out that better upper bounds on $t_{k-1}(\clique{k}{r})$ are known by the work of Lo and Markström~\cite[Corollary 1.5]{LM2015}, but unfortunately their results are stated only for specific combinations of $r, k$ which do not cover all the cases $r \geq k \geq 3$ which we need.

For the proof of Lemma~\ref{lemma:cliquetilingbounds}, we use the powerful framework for hypergraph matchings of Keevash and Mycroft~\cite{KM2015} (see also Han~\cite{H2020} for a simplified proof).

We  need the following standard concepts.
The \emph{down-closure} $\mathcal{H}$ of a $k$-graph $H$ is the usual downward closure of the edge set of $H$.
More generally, for $r\in\mathbb N$, an \emph{$r$-complex} is a hypergraph $J$ whose edges have size at most $r$, and whose edge set is closed under inclusion. For $0 \leq i \leq r$, let $J_i =\{e\in J:|E|=i\}$, and for each {\em $i$-edge} $e \in J_i$,
let $\deg_J(e)=|\{e' \in J_{i+1}: e\subseteq e'\}|$.
The \emph{minimum $i$-degree $\delta_i(J)$ of $J$} is the minimum of $\deg_J(e)$ taken over all $i$-edges $e\in J_i$.

Let $\mathcal{P}=\{V_1, \dotsc, V_t \}$ be a partition of the vertex set of  a $k$-graph $H$.
The \emph{index vector} $\mathbf{i}_{\mathcal{P}}(S) \in \mathbb{Z}^t$ of a set $S \subseteq V(H)$ has coordinates  $\mathbf{i}_{\mathcal{P}}(S)_{i} = |S \cap V_i|$ for $1 \leq i \leq t$.
A vector $\mathbf{v} \in \mathbb{Z}^t$ is a \emph{$k$-vector} if  its coordinates are nonnegative and their sum equals~$k$. Moreover,  if
for some $\mu > 0$,  at least $\mu |V|^k$ edges $e \in E(H)$ satisfy $\mathbf{i}_{\mathcal{P}}(e) = \mathbf{v}$, we call $\mathbf{v}$  a \emph{$\mu$-robust edge vector}.

Let $L^\mu_{\mathcal{P}}(H)$ be the lattice generated by the set of all $\mu$-robust edge-vectors. This lattice is \emph{incomplete} if it does not contain all $k$-vectors in $\mathbb{Z}^t$.
For $1 \leq j \leq t$, let $\mathbf{u}_j$ be the $j$th unit vector, i.e. the vector which has a $1$ in the $j$th coordinate and $0$ otherwise.
A \emph{transferral} is the vector $\mathbf{u}_i - \mathbf{u}_j$ for some $i \neq j$.

\begin{theorem}[{\cite[Theorem 2.12]{KM2015}}] \label{theorem:matchingcomplex}
	Let $1/n \ll \mu \ll \alpha, 1/r$ with $r$ dividing $n$, and let $J$ be an $r$-complex on $n$ vertices with $\delta_i(J) \ge  (1- i/r + \alpha)n$ for each $0\le i<k$. Then one of the following is true:
	\begin{enumerate}
		\item $J_r$ contains a perfect matching, or
		\item $J$ has a \emph{divisibility barrier}, i.e. there is a partition $\mathcal{P}$ of $V(J)$ into at most $r$ parts, each of size at least $\delta_{r-1}(J) - \mu n$, such that $L^{\mu}_{\mathcal{P}}(J_r)$ is incomplete and transferral-free.
	\end{enumerate}
\end{theorem}

\begin{proof}[Proof of Lemma~\ref{lemma:cliquetilingbounds}]
	It suffices to prove that, given any $\alpha >0$, for $n$ sufficiently large and divisible by $r$, any $k$-graph $H$ on $n$ vertices and $\delta(H) \geq (1 - f_k(r) + 3 \alpha) n$ has a $\clique{k}{r}$-factor.
	We select $1/n \ll \mu \ll \alpha, 1/r$ according to the relations in Theorem~\ref{theorem:matchingcomplex}.
	Given a $k$-graph $H$ satisfying the degree conditions, set $J_r:=K_r(H)$ and let $J$ be the down-closure of $J_r$. Our aim is to find a perfect matching in $J_r$. To do so, we first estimate the degree sequence of $J$. By \eqref{eq:r-degree}, we see that for each $k-1 \leq i < r$,
	\begin{align*}
		\delta_i(J) & \geq  n - i - \tbinom{i}{k-1} (n - i - \delta(H))
		\geq n - i - \tbinom{i}{k-1} (n - i - (1 - f_k(r) +3\alpha)n)  \\
		& \ge n - i - \tbinom{i}{k-1}\left( f_k(r) + 2 \alpha \right)n 
		\ge\big( 1 - \tbinom{i}{k-1} f_k(r) + \alpha \big) n \ge  \left(1-\frac ir + \alpha \right)n,
	\end{align*}
	where the last inequality follows from the fact that
	\[ i({\tbinom{r-1}{k-1} + \tbinom{r-2}{k-2}})\ge r\tbinom{i}{k-1},\]
	which in turn is easy to check for $i=0$ and $i=r-1$ and thus holds for all $0\le i< r$ since  ${i}/{\binom{i}{k-1}}$ is decreasing  in $i>0$.
	So the conditions of Theorem~\ref{theorem:matchingcomplex} are fulfilled.
	
	Now, let  $\mathcal{P}$ be any partition of $V(J)$ into $d \leq r$ parts, each of size at least $\delta_{r-1}(J) - \mu n$, such that $L^{\mu}_{\mathcal{P}}(J_r)$ is incomplete (so in particular, $d\ge 2$).
	We only need to show that $\mathcal{P}$ is not  transferral-free. Then, by Theorem~\ref{theorem:matchingcomplex}, $J_r$  has a perfect matching, and we are done.
	
	\cref{corollary:extendtwovertices} shows that for distinct  $V_1, V_2\in\mathcal{P}$ and $v_1 \in V_1, v_2 \in V_2$,
	there are at least $\alpha n^{r-1}$ edges $K$ in $J_{r-1}$ such that $K \cup \{v_1\}$ and $K \cup \{v_2\}$ form edges in $J_r$,
	call these edges  \emph{$(v_1, v_2)$-cliques}.
	At most $c:=\binom{r+d-2}{d-1}$ vectors in $\mathbb{Z}^d$  are $(r-1)$-vectors, that is, vectors  corresponding to the index vector of an $(r-1)$-set in $V(J)$.
	So by the pigeonhole principle, there is an $(r-1)$-vector $\mathbf{i}_{v_1, v_2} \in \mathbb{Z}^d$ such that $\mathbf{i}_{\mathcal{P}}(K) = \mathbf{i}_{v_1, v_2}$ for at least $\alpha c^{-1} n^{r-1}$ many $(v_1, v_2)$-cliques $K$.
	Since $|V_1|, |V_2| \geq \delta_{k-1}(J) - \mu n \geq n/k$,
	there are at least $(n/k)^2$ choices for $(v_1, v_2) \in V_1 \times V_2$.
	By the pigeonhole principle, there is a set $F \subseteq V_1 \times V_2$ with $|F|\ge c^{-1}(n/k)^2$, and an $(r-1)$-vector $\mathbf{i}^\ast \in \mathbb{Z}^d$ such that $\mathbf{i}_{w_1, w_2} = \mathbf{i}^\ast$ for all  $(w_1, w_2) \in F$.    Hence there are at least $|F| \alpha c^{-1} n^{r-1}/n \geq  \alpha n^{r} / (c^2k^2) \geq \mu n^{r}$ edges in~$J_r$ with index vector  $\mathbf{i}^\ast + \mathbf{u}_1$. Thus  $\mathbf{i}^\ast + \mathbf{u}_1$ is a $\mu$-robust edge vector.
	The same holds for $\mathbf{i}^\ast + \mathbf{u}_2$, and hence $(\mathbf{i}^\ast + \mathbf{u}_1) - (\mathbf{i}^\ast + \mathbf{u}_2) = \mathbf{u}_1 - \mathbf{u}_2 \in L^\mu_{\mathcal{P}}(H)$.
	So $L^\mu_{\mathcal{P}}(H)$ is not transferral-free, which is what we needed to show.
\end{proof}

\subsection{Almost spanning factors which avoid bad edges}

Now we carry out step \ref{item:matchingsteps-irregular} of our proof of \cref{lemma:uniformlydensematching}.
This is given by the following lemma, which states that any $k$-graph $G$ of $d$-degree slightly above the threshold guaranteeing an $H$-factor has  an almost spanning $H$-matching  which avoids any small set of `bad edges' fixed beforehand.

\begin{lemma} \label{lemma:Hfactoravoiding}
	Let $1/n \ll a \ll \gamma, 1/r, b$ and $1 \leq d < k \leq r$, let
	$H$ be an $r$-vertex $k$-graph,
	let $G$ be an $n$-vertex $k$-graph and let $I \subseteq G$.
	If $\delta_{d}(G) \geq ( t_d(H) + \gamma ) \binom{n-d}{k-d}$ and $|I| \leq an^k$, then $G \setminus I$ has an $H$-matching which covers all but at most $b n$ vertices.
\end{lemma}

To prove \cref{lemma:Hfactoravoiding} we will use a result by Coulson, Keevash, Perarnau and Yepremyan~\cite{CKPY2020}.
Given an edge-colouring of a $k$-graph $G$, a subgraph $F \subseteq G$ is \emph{rainbow} if all edges of $F$ have different colours.
An edge-colouring of a $k$-graph $G$ on $n$ vertices is \emph{$\mu$-bounded} if, for every colour $\alpha$,
\begin{enumerate}
	\item at most $\mu n^{k-1}$ edges use colour $\alpha$, and
	\item for every $S$ of $k-1$ vertices,
	at most $\mu n$ edges of colour $\alpha$ contain $S$.
\end{enumerate}

\begin{theorem}[{\cite[Theorem 1.2]{CKPY2020}}] \label{theorem:rainbowfactor}
	Let $1/n \ll \mu \ll \gamma, 1/r$ and $1 \leq d < k \leq r$ with $n$ divisible by $r$, let $H$ be a $k$-graph on $r$ vertices,
	and let $G$ be a $k$-graph on $n$ vertices. If
	$\delta_{d}(G) \geq ( t_d(H) + \gamma ) \binom{n-d}{k-d}$,
	then any $\mu$-bounded edge-colouring of $G$ admits a rainbow $H$-factor.
\end{theorem}

\begin{proof}[Proof of \cref{lemma:Hfactoravoiding}]
	Let $\mu$ be   given by \cref{theorem:rainbowfactor} for input $\gamma, r$.
	Set $a = \mu b / (28r)$,
	and let $n$ be sufficiently large. Let $G$ and $H$ be given, as well as $I\subseteq G$ with $|I| \leq an^k$, as in the lemma.
	We need to show that $G \setminus I$ has an $H$-matching which covers all but   
	at most $bn$ vertices.
	First, remove at most $r-1$ arbitrarily chosen  vertices from $G$ to pass to a subgraph $G'$ such that $|V(G')|=m$ is divisible by $r$.
	Set $I'= I\cap G'$ and note that $|I'| \leq a n^k \leq 2 a m^k$.
	
	We define an edge-colouring $c$ of $G'$ as follows.
	Each edge of $I$ gets a colour chosen independently and uniformly at random from a set $C$ with  $|C|= \lceil 14 a m/\mu \rceil$.
	Each edge of $G \setminus I$ gets a fresh and unique colour which is not in $C$.
	
	We claim that, with non-zero probability, $c$ is $\mu$-bounded.
	Since each colour $\alpha \notin C$ is used for exactly  one edge, we only need to check the properties for the remaining colours.
	Let $\alpha \in C$ and let $I_\alpha \subseteq I$ be the edges  of colour $\alpha$.
	Then $\expectation[|I_\alpha|] = |I|/|C| \leq 2a m^k / |C| \le \mu m^{k-1} / 7$.
	Hence, by Lemma~\ref{lemma:concentration}\ref{lemma:happychernoff},
	\begin{align}\label{eins}
		\probability[|I_c| \geq \mu m^{k-1}]
		& \leq 
		\exp(-\mu m^{k-1}) <  \exp(-\mu (n-r)^{k-1}) <\frac{1}{2|C|},
	\end{align}
	where the last inequality holds by the choice of $n$.
	Similarly, 
	let $d_\alpha(S)$ be the number of edges of $I_\alpha$ containing $S$, for each $(k-1)$-set $S$ and $\alpha\in C$.
	Then we see that 
	$\expectation[d_\alpha(S)] \leq m/|C| \leq \mu/(14a) \leq \mu m / 7$,
	and therefore, by Lemma~\ref{lemma:concentration}\ref{lemma:happychernoff},
	\begin{align}\label{zwei}
		\probability[d_c(S) \geq \mu m]
		& \leq 
		\exp(-\mu m) < \frac{1}{2m^{k-1}|C|}.
	\end{align}
	Estimates~\eqref{eins} and~\eqref{zwei}, together with a union bound, show that with non-zero probability, $c$ is $\mu$-bounded, as claimed.
	
	So we can assume $c$ is $\mu$-bounded. Apply \cref{theorem:rainbowfactor} to see that $G'$ has a rainbow $H$-factor~$F$.
	Obtain $F'$ from $F$ by removing all copies of $H$ that use a colour from $C$. Since $F$ is rainbow, there are at most $|C|$ such copies, and thus $F'$ covers all but  at most $r|C|\le 15arm/\mu \leq bm/2$ vertices from $G'$.
	In  $G$, the number of uncovered vertices is at most $bm/2 + r \leq bn$, as required.
\end{proof}

\begin{remark} \label{remark:almost}
	We finish by remarking that we can replace the use of Lemmas \ref{lemma:cliquetilingbounds} and \ref{lemma:Hfactoravoiding} with a different approach.
	Lo and Markström~\cite[Corollary 1.7]{LoMarkstrom2013} showed that for every $\eps > 0, \alpha > 0$, and sufficiently large $n$, every $n$-vertex $k$-graph $H$ with $\delta(H) \geq (1 - \binom{r-1}{k-1}^{-1} + \alpha) n$ admits a collection of vertex-disjoint copies of $K^k_r$ which leave at most $\eps n$ vertices uncovered.
	Then, to prove \cref{lemma:Hfactoravoiding}, one can follow the approach of Ferber and Kwan~\cite[Lemma 3.3]{FerberKwan2022}.
	Partition the vertex set of $H$ randomly into $n/Q$ pieces of size $Q$ each, where $Q$ is a large  constant.
	It is possible to show (cf.~\cite[Lemma 3.4]{FerberKwan2022}) that for almost every set $S$ in such a partition, the  (relative) minimum codegree of
	$(H \setminus I)[S]$ is close to  $\delta(H)$.
	So the result of Lo and Markström  gives an almost-perfect matching in each $(H \setminus I)[S]$. The union of these matchings is an almost-perfect matching in $H \setminus I$.
\end{remark}

\section{Hypergraphs with bounded tree-width}
{In this section we give the proof of Theorem~\ref{theorem:main-treewidth}.
	First, we will need to recall or adapt some facts about absorbers originally used in~\cite{PSS2020},
	then we combine this with the new tools developed in Sections \ref{section:connections}--\ref{section:uniformlydense}.}

\subsection{Absorbing structures}
We now introduce an absorption technique for hypergraphs which will allow us to complete an embedding of almost all of a spanning tree to a full embedding.
We start by introducing our absorbing structure, after a quick definition.
If $H$ is an $r$-graph and $v\in V(H)$, write $H(v)$ for the $(r-1)$-hypergraph consisting of all $f\in\partial_{r-1}(H)$ with $f\cup\{v\}\in H$ {(usually called the \emph{link-graph of $H$ at $v$}).}

\begin{definition}[Absorbing $X$-tuple]
	Let $r\ge 3$ and let $X$ be an $(r-1)$-tree on $h\ge r-1$ vertices, with a fixed valid ordering $x_1, \dotsc, x_{h}$.
	For an $r$-tree $T
	$, we say that an $(h + 1)$-tuple $(v_1, \dotsc, v_{h}, v^\ast)$ of vertices of $T$ is an \emph{$X$-tuple} if
	\begin{enumerate}[\normalfont{(\roman*)}]
		\item $V(T(v^\ast)) = \{ v_1, \dotsc, v_{h} \}$, and
		\item the map $x_i \mapsto v_i$ is a hypergraph isomorphism between $X$ and $T(v^\ast)$.
	\end{enumerate}
	Let $H$ be an $r$-graph on $n$ vertices and let $v_1, \dotsc, v_r\in V(H)$ be  distinct. An \emph{absorbing $X$-tuple for $(v_1, \dotsc, v_r)$} is an $(h + 1)$-tuple $(u_1, \dotsc, u_h, u^\ast)$ of vertices of $H$ such that
	\begin{enumerate}[(A)]
		\item \label{Xtuple:1} $\{ v_1, \dotsc, v_{r-1}, u^\ast \} \in H$, and
		\item \label{Xtuple:2} there exists a copy $\tilde{X}$ of $X$ on $\{ u_1, \dotsc, u_h \}$ such that $\tilde{X} \subseteq H(v_r) \cap H(u^\ast)$.
	\end{enumerate}
	We write $\Lambda_X(v_1, \dotsc, v_r)$ for the set of absorbing $X$-tuples for $(v_1, \dotsc, v_r)$, and  let $\Lambda_X$ be the union of $\Lambda_X(v_1, \dotsc, v_r)$ over all distinct $v_1, \dotsc, v_r\in V(H)$. 
	For an embedding  $\varphi: V(T) \rightarrow V(H)$  and $u_1, \dotsc, u_{h}, u^\ast\in V(H)$,
	we say that $(u_1, \dotsc, u_{h}, u^\ast)$ is \emph{$X$-covered by $\varphi$} if there exists an $X$-tuple $(v_1, \dotsc, v_h, v^\ast)$ of vertices of $T$ such that $\varphi(v^\ast) = u^\ast$ and $\varphi(v_i) = u_i$ for all $i \in [h]$.
\end{definition}
We will need some auxiliary results from~\cite{PSS2020}. {The first result allows us to complete a partial embedding of an $r$-tree to a full embedding of an $r$-tree, provided the partial embedding has sufficiently many $X$-covered absorbing tuples.}

\begin{lemma}[{\cite[Lemma 9.2]{PSS2020}}] \label{lemma:finishingembedding} Let $n\ge h\ge r\ge 3$ and let $0<\delta <\gamma$.
	Let $T$ be an $r$-tree on $n$ vertices with a valid ordering of $V(T)$ given by $v_1, \dotsc, v_n$,  and let $T_0 = T[ \{ v_1, \dotsc, v_{n'} \}]$ be an $r$-subtree of $T$ on $n' \ge (1 - \delta)n$ vertices.
	Let $K$ be an $r$-graph on $n$ vertices, and suppose there are an embedding $\varphi_0:V(T_0)\to V(K)$, an $(r-1)$-tree $X$ on $h$ vertices and a set $\mathcal A \subseteq \Lambda_X$ with the following properties:
	\begin{enumerate}[\normalfont{(\roman*)}]
		\item the tuples in $\mathcal A$ are pairwise vertex-disjoint, \label{itm:absorbingtuples-disjoint}
		\item every tuple in $\mathcal A$ is $X$-covered by $\varphi_0$, and \label{itm:absorbingtuples-covered}
		\item  $|\Lambda_X(v_1, \dotsc, v_r) \cap \mathcal A| \ge \gamma n$ for all tuples such that $\{v_1,\dotsc,v_{r-1}\}\in \partial_{r-1}(K)$ and $v_r\not\in\{v_1,\dotsc,v_{r-1}\}$.  \label{itm:absorbingtuples-many}
	\end{enumerate}	
	Then there exists an embedding of $T$ in $K$.
\end{lemma}

{The second result we need finds a suitable family of vertex-disjoint absorbing tuples.}

\begin{lemma}[Lemma 9.5 in~\cite{PSS2020}] \label{lemma:usingchernoff}
	Let $1/n \ll \beta'\ll\beta, 1/r, 1/h$ with $h \ge r \ge 2$.
	Let $K$ be an $r$-graph  on $n$ vertices and let $X$ be an $(r-1)$-tree on $h$ vertices.
	If $|\Lambda_X(v_1, \dotsc, v_r)| \ge \beta n^{h + 1}$ for every $r$-tuple $(v_1, \dotsc, v_r)$ of distinct vertices of $K$, then there is a set $\mathcal A \subseteq \Lambda_X$ of at most $\beta' n$ pairwise disjoint $(h + 1)$-tuples of vertices of~$K$ such that $|\Lambda_X(v_1, \dotsc, v_r) \cap \mathcal A| \ge \beta \beta' n / 8$ for every $r$-tuple $(v_1, \dotsc, v_r)$ of distinct vertices of~$K$.
\end{lemma}
The third result we need shows that every $r$-tuple has many absorbing tuples in a suitable host graph $H$.
The next lemma and its proof resembles that of Proposition 9.4 in~\cite{PSS2020}, but requires some changes which we give for completeness.

\begin{lemma} \label{proposition:absorbingtuples}
	Let $1/n \ll \beta \ll \alpha, 1/h, 1/k, 1/r$ with $h\ge r \ge k \ge 2$. Let $H$ be an $n$-vertex $k$-graph with $\delta(H)\ge (1-f_k(r)+\alpha)n$,
	let $X$ be an $h$-vertex $(r-1)$-tree, 	and  let $v_1, \dotsc, v_r\in V(H)$ be distinct with
	$\{v_1,\dotsc,v_{r-1}\}\in K_{r-1}(H)$.
	Then $|\Lambda_X(v_1, \dotsc, v_r)| \ge \beta n^{h + 1}$.
\end{lemma}
\begin{proof}
	{
		Let $K = K_r(H)$.
		We construct an absorbing $X$-tuple for $(v_1, \dotsc, v_r)$ in two steps.
		First, we select an arbitrary vertex $u^\ast \in N_K( \{ v_1, \dotsc, v_{r-1} \} ) \setminus \{ v_r \}$. (By our condition on $\delta(H)$ and \eqref{eq:r-degree}, there are at least $\alpha n$ choices for $u^\ast$.)
		Next, consider the $(r-1)$-graph $K' := K( v_r ) \cap K( u^\ast )$.
		By \cref{corollary:extendtwovertices}, $K'$ contains at least $(\alpha n)^{r-1}$ edges.
		Since $X$ is an $h$-vertex $(r-1)$-tree, it is an $(r-1)$-partite $(r-1)$-graph.
		In particular, its Tur\'an density is zero, and so we get many copies of $X$ in $K'$ using a supersaturation argument (see, e.g. \cite[Lemma 2.1, Theorem 2.2]{Keevash2011}).
		More precisely, and using $\beta \ll \alpha$,
		we have at least $2 \beta n^{h}$ copies of $X$ in $K'$, for any choice of $u^\ast$.}
	
	{For each different choice of $u^\ast$ and each different copy of $X$ in $K( v_r ) \cap K( u^\ast )$ we get a different gadget in $\Lambda_X(v_1, \dotsc, v_r)$.
		Thus, with the $\alpha n$ possible ways to choose $u^\ast$, this yields $|\Lambda_X(v_1, \dotsc, v_r)| \ge \beta n^{h + 1}$.}
\end{proof}	

\subsection{Embedding trees}

We need three more results from~\cite{PSS2020} that will help us with embedding $T$ from Theorem~\ref{theorem:main-treewidth} into $K_r(H)$.

{The first lemma ensures that for each small $r$-tree $T$, there exists an $(r-1)$-tree $X$ such that $T$ can be embeded while $X$-covering tuples in some host graph $H$.
	The lemma follows easily by concatenating  Proposition 9.3 and Lemma 9.7 of~\cite{PSS2020}.}

\begin{lemma}\label{lemma:coveringtuples}
	Let $1/n \ll \beta \ll \nu \ll \gamma, 1/h' \ll 1/\Delta, r$ with $\Delta, r\ge 2$, and let $\ell\geq r$.	Let $T$ be an $r$-tree on $\nu n$ vertices with $\Delta_1(T)\leq \Delta$.	Then there exists an $(r-1)$-tree $X$ on $h \leq h'$ vertices with the following property:
	
	If $K$ is any $(\gamma, \ell)$-connecting $r$-graph on $n$ vertices, and $\mathcal A \subseteq \Lambda_X$ consists of at most $\beta n$ pairwise disjoint $(h+1)$-tuples of vertices of $K$ which are $X$-tuples;
	then for any $\mathbf v \in\partial^\circ(K)$ and $\mathbf v'\in \partial^\circ( T)$, there is an embedding $\phi: V(T) \rightarrow V(K)$ such that $\phi(\mathbf v')=\mathbf v$ and every tuple in $\mathcal A$ is $X$-covered by~$\phi$.
\end{lemma}

Given an $r$-graph $H$ on $n$ vertices and $\theta > 0$, we say an edge $e \in H$ is \emph{$\theta$-extensible} if $e$ lies in at least $\theta \binom{n-r}{r}$ copies of $K^{(r)}(2)$ in $H$. {Most edges are extensible:}

\begin{lemma}[{\cite[Lemma 6.5]{PSS2020}}]\label{lem:extensible}
	Let $1/n, \theta \ll \varepsilon, 1/r$.
	In any $r$-graph on $n$ vertices, all but at most $\eps \binom{n}{r}$ edges are $\theta$-extensible.
\end{lemma}

{Finally, the following lemma shows that we can embed any almost-spanning bounded-degree $r$-tree in a connecting $r$-graph, given a reservoir and a uniformly dense perfect matching.
	We can also embed the root of such $r$-tree into a specified extensible edge.}

\begin{proposition}[{\cite[Proposition 10.2]{PSS2020}}]\label{proposition:almost-spanning}
	Let $\Delta, r\ge 2$, 
	let $\theta \ll 1/M$, and 
	let  $1/n \ll \mu \ll \theta \ll \varepsilon, \delta \ll \alpha, \gamma, 1/r, 1/\Delta$.
	Let $\ell\geq r$ and let $K$ be a {$(\gamma,\ell)$-connecting} $r$-graph on $n$ vertices with a $(\delta,\mu)$-reservoir $W$.
	Let $\mathcal M$ be
	an $(\varepsilon,d)$-uniformly dense perfect matching of $K-W$ with $|\mathcal M|\le M$. If $(T,\mathbf x)$ is a rooted $r$-tree with $|V(T)|\le (1-\alpha)n$  and $\Delta_1(T) \leq \Delta$, then $K$ contains $T$. Moreover,  for any $\theta$-extensible edge $e$, any $\mathbf v\in\partial^\circ(K)$ with $\mathbf v\subseteq e$ can be chosen as the image~of~$\mathbf x$. 
\end{proposition}

We remark that in the statements corresponding to \cref{lemma:coveringtuples} and \cref{proposition:almost-spanning} which appear on~\cite{PSS2020}, the hypothesis of being $(\gamma,\ell)$-connecting is actually replaced by the stronger hypothesis of being $(\gamma n,W)$-large.
However, the proofs of the corresponding results only use the fact of being $(\gamma n, W)$-large to ensure that every $(r-1)$-set of host graph $K$ has enough neighbours into the reservoir.
Then we use Lemma~\ref{lemma:strengthenedconnectinglemma} inside the reservoir in order to show that $K[W]$ is $(\gamma/2,\ell)$ connecting for $\ell\ge (2r+1)\lfloor{r/2}\rfloor+2r$.
After this is done the proofs do not need any other alteration, so indeed the statements as written here remain true.

\subsection{Proof of Theorem~\ref{theorem:main-treewidth}}
We can finally prove Theorem~\ref{theorem:main-treewidth}.

\begin{proof}[Proof of Theorem~\ref{theorem:main-treewidth}]
	Given $r, k, \Delta, \alpha$ as in the theorem, we choose 
	\[1/n_0 \ll \mu \ll\theta\ll 1/M_0 \ll\gamma,d \ll  \varepsilon\ll \delta'\ll \delta\ll\alpha, 1/r, 1/\Delta.\]
	We are given $G$, $T$ and $H$.
	Since $G \subseteq \partial_k (T)$, in order
	to show that $G \subseteq H$ it is clearly enough to show that $T \subseteq K_r(H)$, which is our task from now on.
	
	We begin by choosing the useful structures in the $r$-tree $T$.
	{Choose an arbitrary root $\mathbf{r}$ of $T$,
		and find a layering $(T,\mathbf{r},\mathcal{L})$ using \cref{lemma:flattening}.
		Now, use \cref{lemma:smallsubtree} to find $\mathbf{x} \in \partial^\circ(T)$, and an $r$-subtree $T' = T_{\mathbf{x}}$ of $T$, such that $|T'|= \nu n$, where $\mu \ll \nu \ll \delta'$.}
	{Apply \cref{lemma:coveringtuples} to $T'$ to obtain an $(r-1)$-tree $X$ with the properties in that statement and say $X$ has $h$ vertices.}
	
	Now we find useful structures in $K_r(H)$. We begin by finding the absorbing tuples.
	Use Lemmas~\ref{lemma:usingchernoff} and~\ref{proposition:absorbingtuples} with suitable $\beta, \beta'$ fulfilling $\delta\ll \beta'\ll\beta\ll\alpha$ to find  a set $\mathcal A \subseteq \Lambda_X$ of at most $\beta' n$ pairwise disjoint $(h + 1)$-tuples of vertices of~$K_r(H)$ such that $|\Lambda_X(v_1, \dotsc, v_r) \cap \mathcal A| \ge \beta \beta' n / 8>\delta n$ for every $r$-tuple $(v_1, \dotsc, v_r)$ of {distinct vertices such that $\{v_1,\dotsc,v_{r-1}\}\in K_{r-1}(H)$.}
	
	Next, we find an extensible tuple and a reservoir in $K_r(H)$.
	Let $H' = H - V(\mathcal{A})$.
	Since $|V(\mathcal{A})| \leq (h+1) \beta' n$ and $\beta' \ll \alpha$,
	we have $\delta(H') \geq \delta(H) - \alpha n/100$.
	Next, we use Lemma~\ref{lem:extensible} to find a $\theta$-extensible edge $e_0$ in~$K_r(H')$ and fix an $(r-1)$-subset $f_0\subseteq e_0$.
	By Lemma~\ref{lem:reservoir}, $K_r(H')$ has a $(\gamma,\mu/2)$-reservoir $W'$.
	Now let $W = W' \setminus f_0$ and $H'' = H - W$.
	Then $\mathcal{A}, e_0 \subseteq H''$,
	and $W$ is a $(\gamma, \mu)$-reservoir for $K_r(H')$.
	Note that $\delta(H'') \ge \delta(H)-\alpha n/50$.
	
	Choose an $r$-tree $T''$ with $T'\subseteq T''\subseteq T$ and $(1-\delta/2)n\ge |T''|\ge (1-\delta)n$ (this can be done by iteratively removing leaves from $T$).
	{If there is an embedding $\varphi_0:V(T'')\to V(K_r(H))$ such that every tuple in $\mathcal{A}$ is $X$-covered by $\varphi_0$, then Lemma~\ref{lemma:finishingembedding} implies that $T\subseteq K_r(H)$, as desired.}
	So it only remains to find $\varphi_0$ of this form.
	
	Since $\delta(H'') \ge \delta(H)-\alpha n/50$,
	\cref{lemma:connectivity} implies that $K_r(H'')$ is $(\alpha/4, \ell)$-connecting for some $\ell$ with $1/\ell \ll 1/k, 1/r$.
	The conditions required by \cref{lemma:coveringtuples} are met by $T', X$ and $K_r(H'')$.
	So we obtain an embedding $\varphi_1$ of $T'$ such that $\phi(\mathbf{x}) = e_0$,
	and such that  every tuple in $\mathcal{A}$ is $X$-covered by $\varphi_1$.
	Set $H^\ast:=H-\varphi_1(V(T')\setminus e_0)$.
	Note that $\delta(H^\ast)\ge \delta(H)-\alpha n/25$.
	
	{By Lemma~\ref{lemma:uniformlydensematching}, $H^\ast-W$ has a subgraph $H''$ with $|H''|\ge (1-\delta')|H^\ast-W|\ge (1-2\delta')n$ such that $K_{r}(H'')$ admits an $(\varepsilon,d)$-uniformly dense matching $\mathcal M$ with $|\mathcal M|\le M_0$.
		By Proposition~\ref{proposition:almost-spanning} and since $\delta '\ll\delta$, there is an embedding $\varphi_2$ of $T''\setminus (T'-f_0)$ in $K_{r}(H^\ast)$ such that $\varphi_2(\mathbf{x})=\varphi_1(\mathbf{x})$.
		Together, $\varphi_1$ and $\varphi_2$ constitute the embedding $\varphi_0$ of $T''$ we were looking for.}
\end{proof}

\section{Powers of tight Hamilton cycles}
In this section, we prove \cref{theorem:main-powcycles}.
We start with a definition.

\begin{definition}[$v$-path-absorbing]
	Let $r\ge k\ge 2$, let $H$ be a $k$-graph, and let $v\in V(H)$. Call a $(2r-2)$-tuple
	$(v_1,\dots,v_{2r-2})\in (V(H)\setminus \{v\})^{2r-2}$ of distinct vertices  \emph{$v$-path-absorbing} if both $v_1\dots v_{2r-2}$ and $v_1\dots v_{r-1}vv_{r}\dots v_{2r-2}$ consist of $(r-k+1)$th powers of a tight Hamilton path.
\end{definition}

We now show that in the situation of Theorem~\ref{theorem:main-powcycles}, there is always a set of pairwise disjoint tuples which are very useful for absorbing vertices.

\begin{lemma}\label{lem:absorbingtuples:cycles}
	Let $1/n\ll\beta\ll \nu \ll \alpha, 1/r$ and $r\ge k\ge 2$, and let $H$ be an $n$-vertex $k$-graph with $\delta(H)\ge (1-f_k(r)+\alpha)n$.
	Then there is a set $\mathcal A\subseteq (V(H))^{2r-2}$ of  pairwise vertex-disjoint  tuples such that $|\mathcal A|\le \nu n$, and for each $v\in V(H)$ there are at least $\beta n$ $v$-path-absorbing tuples in $\mathcal A$. 
\end{lemma}
\begin{proof}
	For each $v \in V(H)$, let $\mathcal A_v$ be the set of $v$-path-absorbing $(2r-2)$-tuples.
	We start by showing that
	\begin{equation}\label{manyabsorbingtuplesforv}
		\text{for each $v\in V(H)$, $|\mathcal A_v|\ge (\alpha n)^{2r-2}$.}
	\end{equation}
	{Indeed, let $v \in V(H)$ be arbitrary.
		Initially, let $v_2, v_3, \dotsc, v_{k-1}$ be arbitrary distinct vertices, distinct from $v$ as well.
		Having chosen $v_2, \dotsc, v_{j}$ with $k-1 \leq j < r$, equation \eqref{eq:r-degree} shows that there are at least $\alpha n$ ways to choose $v_{j+1}$ such that $\{v, v_2, \dotsc, v_{j+1} \} \in K_{j+1}(H)$.
		At the end of this process we have defined $v_2, \dotsc, v_r$ such that $\{ v_2, \dotsc, v_{r-1}, v,  v_{r} \} \in K_{r}(H)$.
		Note that $\{v_2, \dotsc, v_{r-1}, v_r\}$ and $\{v_2, \dotsc, v_{r-1}, v \}$ belong to $K_{r-1}(H)$ and intersect in $r-2$ vertices. So, by \cref{lemma:crucialbound}, there are at least $\alpha n$ choices for a vertex $v_1$ such that both $\{ v_1, v_2, \dotsc, v_r \}$ and $\{ v_1, v_2, \dotsc, v_{r-1}, v \}$ are in $K_{r}(H)$.}
	
	{Next, suppose we have defined $v_1, \dotsc, v_j$ for some $r \leq j < 2r - 2$ such that both the sets $\{ v_{j-r+1}, \dotsc, v_j \}$ and $\{ v, v_{j-r+2}, \dotsc, v_j \}$ are in $K_r(H)$.
		Then the two $(r-1)$-sets $\{ v_{j-r+2}, v_{j-r+3}, \dotsc, v_j \}$ and $\{ v, v_{j-r+3}, \dotsc, v_j \}$ are in $K_{r-1}(H)$ and they intersect in $r-2$ vertices.
		\cref{lemma:crucialbound} then implies the existence of at least $\alpha n$ choices of $v_{j+1}$ such that both $\{ v_{j-r+2}, v_{j-r+3}, \dotsc, v_j, v_{j+1} \}$ and $\{ v, v_{j-r+3}, \dotsc, v_j, v_{j+1} \}$ are edges in $K_r(H)$.
		At the end of this process we have found $(v_1, \dotsc, v_{2r-2})$ which is a $v$-path-absorbing tuple.
		Since there were at least $\alpha n$ choices for each $v_i$, we get \eqref{manyabsorbingtuplesforv}.}
	
	Let $\beta \ll c\ll \nu$ and set $p=cn^{-2r+1}$. Consider $\mathcal A^\cup:=\bigcup_{v\in V(H)}\mathcal A_v$. Construct a set $\mathcal A^p\subseteq \mathcal A^\cup$ by independently choosing each element of $\mathcal A^\cup$ with probability~$p$. Note that $\expectation[|\mathcal A^p|]\le cn$ and $\expectation[\mathcal A^p\cap \mathcal A_v]\ge \alpha^{2r-2}c n$ for all $v\in V(H)$. So Lemma~\ref{lemma:concentration}\ref{lemma:happychernoff}--\ref{lemma:sadchernoff}, and the union bound, imply that
	\begin{enumerate}
		\item $\probability[|\mathcal A^p|\ge 7cn]\le \exp (-7c n)<1/4$, and
		\item $\probability[|\mathcal A^p\cap \mathcal A_v|\le \alpha^{2r-2}c n/2 \text{ for some }v\in V(H)]\le n\exp (-\alpha^{2r-2}c n/12)<1/4$,
	\end{enumerate}
	for $n$ sufficiently large.
	
	Let $\mathcal{X} \subseteq (\mathcal{A}^\cup)^2$ be the pairs of absorbing tuples in $\mathcal{A}^\cup$ which share at least one vertex.
	There are clearly at most $(2r-2)n^{4r-5}$ such pairs in $\mathcal{X}$.
	Let $\mathcal{X}^p \subseteq \mathcal{X}$ be the set of pairs such that both of which survived in $\mathcal{A}^p$.
	Clearly $\expectation[|\mathcal{X}^p|] = |\mathcal{X}|p^2 \leq (2r-2)n^{4r-5} p^2 = (2r-2)c^2 n$.
	By Markov's inequality, we have that
	\begin{enumerate}[resume]
		\item $\probability[|\mathcal X^p|\ge (4r-4)c^2n]\le 1/2$.
	\end{enumerate}
	Therefore, there exists a realisation $\mathcal A$ of $\mathcal A^p$ such that $|\mathcal A|\le 7cn\le \nu n$, $|\mathcal A \cap \mathcal A_v|\ge \alpha^{2r-2}cn/2$ for all $v\in V(H)$, and $|\mathcal X^p| < (4r-4)c^2n$.
	By removing one tuple from each pair $\mathcal X^p$, we arrive at a collection of vertex-disjoint tuples $\mathcal A' \subseteq \mathcal A$ such that $|\mathcal A'| \leq \nu n$ and, for each $v \in V(H)$,
	$|\mathcal A' \cap \mathcal A_v| \geq |\mathcal A \cap \mathcal A_v| - |\mathcal X^p| \geq \beta n$, as required.
\end{proof}

\begin{lemma}[Absorbing lemma: power of cycles]\label{lemma:absorbing:cycles}
	Let $1/n \ll 1/r$ and $2\le k\le r$, let $0<\delta<\beta$ and let $n'\ge (1-\delta)n$. Let $X$ be the $(r-k+1)$th power of a tight $k$-cycle on $n'$ vertices, and let $H$ be an $n$-vertex $k$-graph with $X\subseteq H$. Let $\mathcal A\subseteq V(H)^{2r-2}$, and assume that 
	\begin{enumerate}
		\item each tuple in $\mathcal A$ is a segment of length $2r-2$ in $X$, 
		\item the elements of $\mathcal A$ are pairwise vertex-disjoint, and
		\item for all $v\in V(H)$ there are at least $\beta n$ $v$-path-absorbing tuples in $\mathcal A$.
	\end{enumerate}
	Then $H$ contains the $(r-k+1)$th power of a tight Hamilton cycle.
\end{lemma}

\begin{proof}
	Write $V(H-X)=\{x_1,\dotsc, x_{n-n'}\}$. For $0\le i\le n-n'$, we will inductively construct a graph $X_i$ and a collection of vertex-disjoint tuples $\mathcal A_i\subseteq \mathcal A$ such that 
	\begin{enumerate}
		\item\label{lemma:absorbing:subgraph} $X_i$ is a copy of $C_{k,n'+i}^{r-k+1}$, 
		\item\label{lemma:absorbing:segment} every tuple in $\mathcal A_i$ is a segment of length $2r-2$ in $X_i$, and
		\item\label{lemma:absorbing:tuples} for each $x\in\{x_{i+1},\dots, x_{n-n'}\}$ there are at least $\beta n-i$ $x$-path-absorbing tuples in~$\mathcal A_i$.
	\end{enumerate}
	In the final step $i=n-n'$ this construction will give the desired $(r-k+1)$th power of a tight Hamilton cycle.
	
	Set $X_0=X$ and $\mathcal A_0=\mathcal A$. Now, let $0\le i\le m-1$ and assume  we have found $X_i$ and~$\mathcal A_i$. Let $v_1,\dotsc,v_{n'+i}$ be a enumeration of $V(X_i)$.  We construct $X_{i+1}$ by adding $x_{i+1}$ to $X_{i}$ using an $x_{i+1}$-path-absorbing tuple. Because of~\ref{lemma:absorbing:tuples} for $i$, we know there are at least $\beta n-i\ge (\beta-\delta)n>0$ suitable tuples in $\mathcal A_i$, and thus by~\ref{lemma:absorbing:segment} for $i$,  there is  $j\in [n'+i]$ such that $\mathbf v=(v_{j+1},\dots,v_{j+2r-2})\in\mathcal A_i$ is an $x_{i+1}$-absorbing tuple. Set $\mathcal A_{i+1}=\mathcal A_i\setminus \{\mathbf v\}$. 
	Then $X_{i+1}=v_1\dots v_jv_{j+r-1}x_{i+1}v_{j+r}\dots v_{j+2r-2}\dots v_{n'+i}$ is a copy of $C_{k,n'+i+1}^{r-k+1}$, which is as desired for (i), while (ii) and (iii) clearly hold by construction.
\end{proof}
We are now ready to prove Theorem~\ref{theorem:main-powcycles}.

\begin{proof}[Proof of Theorem~\ref{theorem:main-powcycles}] We begin by letting $\ell=5kr!$ and choosing constants
	\[1/{n_0}\ll \mu\ll \gamma\ll\delta\ll \beta\ll\nu\ll
	\alpha/(r\ell)\]
	The proof is divided into three main steps. \\
	
	\noindent\emph{Step 1: Finding an absorbing path.} Let $U\subseteq V(H)$ be an $(\gamma, \mu)$-reservoir for $V(H)$ given by Lemma~\ref{lem:reservoir}. Since $|U|\le 2\gamma n$ and $\gamma\ll \alpha$, we may use Lemma~\ref{lem:absorbingtuples:cycles} in $H \setminus U$ to find a family $\mathcal A=\{\mathbf v_1,\dotsc, \mathbf v_t\}\subseteq (V(H)\setminus U)^{2r-2}$ of  $t\le \nu n$ vertex-disjoint $(2r-2)$-tuples such that for each vertex $v\in V(H)$ there are at least $\beta n$  $v$-path-absorbing tuples in $\mathcal A$. Let  $\mathbf v_i^1,\mathbf v_i^2\in (V(H)\setminus U)^{r-1}$ with $\mathbf v_i=(\mathbf v_i^1,\mathbf v_i^2)$ for $i\in \{1, \dotsc, t\}$.
	
	We now construct
	in $H-U$ the $(r-k+1)$th power of a tight path covering every tuple in $\mathcal A$ by successively applying Lemma~\ref{lemma:connectivity:cycles}. For $1\le i<t$, let us suppose that we have constructed a tight path in $K_r(H)$ covering $\mathbf{v}_1,\dotsc,\mathbf{v}_i$ such that for each $1\le j\le i-1$ there is a path $P_j$ of length at most $\ell$ connecting $\mathbf v_j^2$ with $\mathbf{v}_{j+1}^1$. Since we need to avoid $U$ and the at most $kt\ell\ll\alpha n/4$ already used vertices, we may apply Lemma~\ref{lemma:connectivity:cycles} to find a tight path $P_i$ of length $\ell$ in $K_r(H-(U\cup V(P_1\cup \dots \cup P_{i-1})))$ such that $\sta{P_i}=\mathbf{v}_i^2$ and $\ter{P_i}=\mathbf v_{i+1}^1$.  Then $P=P_1\cup\dotsc\cup P_t$ is the $(r-k+1)$th power of a tight path covering each tuple of $\mathcal A$.\\

	\noindent\emph{Step 2: Finding an almost spanning $(r-k+1)$th power of a tight cycle.} Let $H'=H-(U\cup V(P))$ and note that 
	$\delta(H')\ge\delta(H)-|U|-|V(P)|\ge (1-f_k(r)+\alpha/2)n.$ Using Theorem~\ref{theorem:main-treewidth}, we can thus find in $H'$ a spanning $k$-subgraph $P'$ which is the $(r-k+1)$th power of a tight path. Since every element in $\partial_{k-1}(H)$ has at least 
	$(\delta(H)/n-\mu)|U|\ge (1-f_k(r)+\alpha/2)|U|$
	neighbours in $U$, we can greedily find $\mathbf u_1,\mathbf u_2, \mathbf w_1,\mathbf w_2\in U^{r-1}$ such that $\mathbf u_1P\mathbf w_1$ and $\mathbf u_2P'\mathbf w_2$ are   $(r-k+1)$th powers  of tight paths. Since 
	$\delta(H[U])\ge \left({\delta(H)}/n-\mu\right)|U|\ge (1-f_k(r)+\alpha/2)|U| $,
	we can use Lemma~\ref{lemma:connectivity:cycles} to find two disjoint paths $Q,Q'$ in $K_{r}(H[U])$ such that $\sta{Q}=\mathbf u_1$ and $\ter{Q}=\mathbf u_2$, and $\sta{Q'}=\mathbf w_1$ and $\ter{Q'}=\mathbf w_2$. Thus we have found the $(r-k+1)$th power of a tight cycle of length $n'\ge n-|U|\ge (1-\delta)n$. \\
	
	\noindent\emph{Step 3: Finishing the embedding.} In the previous step we  constructed the $(r-k+1)$th power of a tight cycle on $n'\ge (1-\delta)n$ vertices covering all tuples in $\mathcal A$. So Lemma~\ref{lemma:absorbing:cycles} implies that~$H$ also contains a spanning $(r-k+1)$th power of a Hamilton tight cycle.  
\end{proof}
\section{A variant of  \cref{theorem:main-treewidth}}\label{sec:twvariant}
In this short section, we show the following variant of  \cref{theorem:main-treewidth}.
\begin{theorem} \label{other:main-treewidth}
	For any $r\ge k\ge 2$,  $\Delta > 0$ and $\alpha > 0$, there exists $n_0$ such that  for all $n \geq n_0$, any  $n$-vertex $k$-graph $G$ of tree-width less than $r$ 
	with $\treewidth(G)<r$, and any $n$-vertex $k$-graph  $H$ with 
	$\delta(H) \geq (1 - (\binom{r'-1}{k-1} + \binom{r'-2}{k-2})^{-1} + \alpha)n, $
	where $r':=36(r+1)(k-1) \Delta$, we have $G \subseteq H$.
\end{theorem}

For the proof, we will need a definition and a lemma. Given a graph $G$ and a clique $K_r$ on $r$ vertices, the graph $G \boxtimes K_r$ is obtained from $G$ by replacing every vertex by a clique on $r$ vertices, and each edge is replaced by a complete bipartite graph between the two corresponding cliques. 
The following result was proven by Distel and Wood~\cite{DistelWood2022}.
%

\begin{lemma} \label{lemma:treepartitionwidth-improved}
	Let $G$ be a graph of treewidth $t \geq 1$ with maximum degree at most $\Delta$.
	Then there exists a tree $T$ of degree at most $6 \Delta$ such that $G \subseteq T \boxtimes K_{18(t+1)\Delta}$.
\end{lemma}

\begin{proof}[Proof of Theorem~\ref{other:main-treewidth}]
	Let $G$ be an $n$-vertex $k$-graph with $\treewidth(G) \leq r$ and $\Delta_1(G) \leq \Delta$.
	Then $\Delta(\partial_2(G)) \leq (k-1) \Delta$ and $\treewidth(\partial_2(G))=\treewidth(G)$.
	Apply Lemma~\ref{lemma:treepartitionwidth-improved} to $\partial_2(G)$ to see that $\partial_2(G) \subseteq T \boxtimes K_{18(r+1)(k-1) \Delta}$ for some tree $T$ of maximum degree at most $6 (k-1) \Delta$.
	It is not overly hard to see that $T \boxtimes K_{18(r+1)(k-1)}$ (and therefore $G$) admits an $(r'+1)$-tree $T'$ with $r' < 36(r+1)(k-1) \Delta$ and $\Delta_1(T')
	\leq \Delta(T) 18(r+1)(k-1) \Delta
	\leq 108(r+1)(k-1)^2 \Delta^2$.
	Combined with \cref{theorem:main-treewidth} this yields a copy of $G$ in the $n$-vertex $k$-graph~$H$.
\end{proof}

\appendix 
\section*{Appendix}

\section{Proof of \cref{lemma:uniformlydensematching}} \label{appendix:regularity}

In this Appendix we give the proof of \cref{lemma:uniformlydensematching}.
We start by introducing hypergraph regularity,
which we  use  to find uniformly dense tuples (\cref{proposition:uniformlydensematchingviaregularity}),
which in turn allows us to prove the lemma.

We follow the exposition by Allen, B\"ottcher, Cooley and Mycroft~\cite{ABCM2017} and their `Regular Slice Lemma',
which is a convenient formulation of the hypergraph regularity machinery, previously developed by many authors~\cite{Gowers2007, RodlSkokan2004, NagleRodlSchacht2006}.

Recall that  the down-closure $\mathcal{H}$ of a $k$-graph $H$ was defined in Section~\ref{section:cliquefactors}.
We say  a set $S$ of $k+1$ vertices is \emph{supported on} $\mathcal{H}$ if every $k$-set $S'\subseteq S$  is an edge of $\mathcal{H}$.
Given  an $i$-partite $(i-1)$-graph $H_{i-1}$, observe that $K_{i}(H_{i-1})$ is an $i$-partite $i$-graph whose edges are all the $i$-sets supported on $H_{i-1}$.

\begin{definition}
	Let $i\geq 2$, let $H_i$ be an $i$-partite $i$-graph and let $H_{i-1}$ be an $i$-partite $(i-1)$-graph on a common vertex set $V$.
	The \emph{relative density of $H_i$ with respect to $H_{i-1}$} is defined to be $0$ if $|K(H_{i-1})|=0$ and
	$$d(H_{i}|H_{i-1})=\dfrac{|K_i(H_{i-1})\cap H_i|}{|K_i({H_{i-1})|}}$$
	otherwise.
	More generally, if $\mathbf{Q}=(Q_1,\dots, Q_s$) is a collection of not necessarily disjoint subgraphs of $H_{i-1}$, we define $K_i(\mathbf{Q})=\bigcup_{j=1}^s K_i(Q_j)$.
	When this set is non empty, we define
	$$d(H_i|\mathbf{Q})=\dfrac{|K_i(\mathbf{Q})\cap H_i|}{|K_i(\mathbf{Q})|},$$
	and we set $d(H_i|\mathbf{Q}) = 0$ otherwise.
\end{definition}

\begin{definition}
	Let $i\geq 2$, let $H_i$ be an $i$-partite $i$-graph and let $H_{i-1}$ be an $i$-partite $(i-1)$-graph on a common vertex set $V$.
	We say that \emph{$H_i$ is $(d_i,\varepsilon,s)$-regular} with respect to $H_{i-1}$ if $d(H_i|\mathbf{Q})=d_i\pm\varepsilon$ for every $s$-set $\mathbf{Q}$ of subgraphs of $H_{i-1}$ such that $|K_i(\mathbf{Q})|>\varepsilon|K_i(H_{i-1})|$.
	In the case that $s=1$, it will be referred simply as \emph{$(d_i,\varepsilon)$-regularity}.
	
	An $i$-graph $G$ whose vertex set contains that of $H_{i-1}$  is called \emph{$(d_i,\varepsilon,s)$-regular with respect to $H_{i-1}$}
	if the $i$-partite subgraph of $G$ induced by the vertex classes of $H_{i-1}$ is $(d_i,\varepsilon,s)$-regular with respect to $H_{i-1}$.
	We call $d_i$ the \emph{relative density of $G$ with respect to $H_{i-1}$}.
\end{definition}

Recall that when $\mathcal{J}$ is a $k$-complex, $\mathcal{J}_i=\{e \in \mathcal{J}: |e|=i\}$.
If $\mathcal{J}$ is $\mathcal{P}$-partite for some partition $\mathcal{P}$, we refer to $\mathcal{P}$ as the \emph{ground partition} of $\mathcal{J}$ and to the parts of $\mathcal{P}$ as the \emph{clusters} of $\mathcal{J}$.
When $\mathcal{J}$ is $r$-partite with vertex classes $V_1,\dots,V_r$ and $A\subseteq [r]$,
we denote $V_A=\bigcup_{i \in A}V_i$.

\begin{definition}
	Let $\mathcal{H}$ be an $r$-partite $(k-1)$-complex on vertex classes $V_1,\dots, V_r$, where $r \geq k-1\geq 2$.
	We say that $\mathcal{H}$ is \emph{$(d_{k-1},\dots,d_2,\varepsilon)$-regular} if
	for any $2\leq i \leq k-1$ and any $A\in{{[r]} \choose i}$, the induced subgraph $\mathcal{H}_i[V_A]$ is $(d_i,\varepsilon)$-regular with respect to $\mathcal{H}_{i-1}[V_A]$.
	For a $(k-1)$-tuple $\mathbf{d}=(d_k,\dots,d_2)$ we write \emph{$(\mathbf{d},\varepsilon)$-regular} to mean $(d_k,\dots,d_2,\varepsilon)$-regular.
\end{definition}
\begin{definition}
	We say that a $(k-1)$-complex $\mathcal{J}$ is \emph{$(t_0,t_1,\varepsilon)$-equitable} if
	\begin{enumerate}
		\item $\mathcal{J}$ has a ground partition into $t$ parts of equal size, where $t_0\leq t\leq t_1$.
		\item There exists a density vector $\mathbf{d}=(d_{k-1},\dots,d_2)$ such that for each $2\leq i\leq k-1$ we have $d_i\geq \frac{1}{t_1}$, $\frac{1}{d_i}\in\mathbb{N}$ and the $(k-1)$-complex $\mathcal{J}$ is $(\mathbf{d},\varepsilon)$-regular.
	\end{enumerate}
	We write \emph{$(\cdot, \cdot, \varepsilon)$-equitable} if we do not need to state $t_0,t_1$ explicitly.	\end{definition}

\begin{definition}
	Let $\mathcal{J}$ be a partite $(k-1)$-complex, let $X$ be 
	a $k$-set of clusters of $\mathcal{J}$,
	and let $G$ be a $k$-graph on $V(\mathcal{J})$. We say that $G$ is \emph{$(\varepsilon_k,s)$-regular with respect to $X$ in $\mathcal{J}$} if there exists some $d$ such tat $G$ is $(d,\varepsilon_k,s)$-regular with respect to $\mathcal{J}_{k-1}[V_X]$.
	We also write $d_\mathcal{J}^*(X)$ for the relative density of $G$ with respect to $\mathcal{J}_{k-1}[V_X]$,
	or simply $d^*(X)$ when $\mathcal{J}$ is clear from the context.
	
	Define the \emph{weighted reduced $k$-graph $R_\mathcal{J}(G)$} as the complete weighted $k$-graph whose vertices are the clusters of $\mathcal{J}$ and where each edge $X$ has weight $d^*(X)$. When $\mathcal{J}$ is clear from the context we often simply write $R(G)$ instead of $R_\mathcal{J}(G)$.
\end{definition}

\begin{definition}
	Given $\varepsilon,\varepsilon_k>0, s,t_0,t_1\in\mathbb{N}$, a $k$-graph $G$ and a $(k-1)$-complex $\mathcal{J}$ on $V(G)$, we call $\mathcal{J}$ a \emph{$(t_0,t_1,\varepsilon,\varepsilon_k,s)$-regular slice} for $G$ if $\mathcal{J}$ is $(t_0,t_1,\varepsilon)$-equitable and $G$ is $(\varepsilon_k,s)$-regular with respect to all but at most $\varepsilon_k{t \choose k}$ of the $k$-sets of clusters in $\mathcal{J}$, where $t$ is the number of clusters of $\mathcal{J}$.
\end{definition}

\begin{definition}
	Let $G$ be a $k$-graph on $n$ vertices. Given a set $S\subseteq V(G)$ of size $k-1$,
	the \emph{relative codegree $\overline{\deg}(S;G)$ of $S$ with respect to $G$} is defined to be
	\[\overline{\deg}(S;G)=\dfrac{|\{e \in G: S\subseteq e\}|}{n-k+1}.\]
	We extend the definition to weighted graphs $G$ by replacing $|\{e \in G: S\subseteq e\}|$ with the sum of the weights of the edges containing $S$.
	Moreover, if $\mathcal{S}$ is a collection of $(k-1)$-sets in $V(G)$ then $\overline{\deg}(\mathcal{S};G)$ is defined to be the average of $\overline{\deg}(S;G)$ over all sets $S \in \mathcal{S}$.
\end{definition}

\begin{lemma}[Regular Slice Lemma~{\cite[Lemma 10]{ABCM2017}}] \label{lemma:regularslices}
	Let $k\geq 3$ be a fixed integer. For all $t_0\in\mathbb N$, $\varepsilon_k>0$ and all functions $s:\mathbb{N}\to\mathbb{N}$ and $\varepsilon:\mathbb{N}\to(0,1]$, there are integers $t_1$ and $n_0$ such that the following holds for all $n\geq n_0$ which are divisible by $t_1!$.
	
	Let $V$ be a set of size $n$ and suppose that $G$ is a $k$-graph on $V$. Then there exists a $(k-1)$-complex $\mathcal{J}$ on $V$ which is a $(t_0,t_1,\varepsilon(t_1),\varepsilon_k,s(t_1))$-regular slice for $G$, and for each set $Y$ of $k-1$ clusters of $\mathcal{J}$, we have
	\begin{equation*}
		\left|\overline{\deg}(Y;R(G))-\overline{\deg}(\mathcal{J}_{k-1}[V_Y];G)\right|<\varepsilon_k.
	\end{equation*}
\end{lemma}

\begin{definition}
	Let $G$ be a $k$-graph and let $\mathcal{J}$ be a $(t_0,t_1,\varepsilon,\varepsilon_k,s)$-regular slice for $G$.
	We define two $k$-graphs $R_d(G)$ and $I(G)$ whose vertex set is the clusters of $\mathcal{J}$.
	For $d>0$, the \emph{$d$-reduced $k$-graph $R_d(G)$} is the $k$-graph whose edges are all the $k$-sets $X$ of clusters of $\mathcal{J}$ such that $G$ is $(\varepsilon_k,s)$-regular with respect to $X$ and $d^*(X)\geq d$.
	The \emph{irregularity $k$-graph} $I(G)$ is the $k$-graph whose edges are all the $k$-sets $X$ of clusters of $\mathcal{J}$ such that $G$ is not $(\varepsilon_k, s)$-regular.
\end{definition}

Note that $R_d(G)$ and $I(G)$ are edge-disjoint $k$-graphs on the same vertex set.
The next lemma states that for regular slices $\mathcal{J}$ from the Regular Slice Lemma,
the codegree of $G$ is almost preserved in $R_d(G) \cup I(G)$.
\begin{lemma}[{\cite[Lemma 12]{ABCM2017}}]
	Let $G$ be a $k$-graph and let $\mathcal{J}$ be a $(t_0,t_1,\varepsilon,\varepsilon_k,s)$-regular slice for $G$.        For any set $Y$ of $k-1$ clusters of $\mathcal{J}$ we have
	\begin{equation}
		\overline{\deg}(Y;R_d(G) \cup I(G))\geq \overline{\deg}(Y;R(G))-d. \label{equation:regularslice-dreduced+irregular}
	\end{equation}
\end{lemma}

Note that if $\mathcal{J}$ is a partite $(d_{k-1}, \dotsc, d_2, \varepsilon)$-regular $(k-1)$-complex and $G$ is supported on $\mathcal{J}_{k-1}$,
then $\mathcal{J} \cup G$ is a $k$-complex which is partite with the same partition of $\mathcal{J}$.
If for all $k$-sets $X$ of $\mathcal{J}$,
$G$ is $(\varepsilon_k,s)$-regular with respect to $X$ and $d_\mathcal{J}^*(X) \geq d_k$,
then we say that $\mathcal{J} \cup G$ is a \emph{$(d_k, d_{k-1}, \dotsc, d_2, \varepsilon, \varepsilon_k, s)$-regular complex}.

\begin{lemma}[Counting Lemma~{\cite[Lemma 27]{ABCM2017}}] \label{lemma:counting}
	Let $k,r,s,m$ be positive integers, and let $d_2, \dotsc, d_k, \varepsilon, \varepsilon_k, \beta$ be positive constants such that $1/d_i \in \mathbb{N}$ for any $2 \leq i \leq k-1$ and
	\[ \frac{1}{m} \ll \frac{1}{s}, \varepsilon \ll \varepsilon_k, d_2, \dotsc, d_{k-1} \qquad \text{and} \qquad \varepsilon_k \ll \beta, \frac{1}{r}. \]
	Let $\clique{k}{r}$ be a $k$-uniform clique with vertex set $[r]$,
	let $\mathcal{J}$ be an $r$-partite $(k-1)$-complex with $r$ clusters $V_1, \dotsc, V_r$, each of size $m$,
	and also let $G$ be a $k$-graph which is supported on $\mathcal{J}_{k-1}$.
	Suppose $G \cup \mathcal{J}$ is a $(d_k, d_{k-1}, \dotsc, d_2, \varepsilon, \varepsilon_k, s)$-regular complex.
	Then the number of copies of $\clique{k}{r}$ in $G$ such that $i \in V_i$ for all $1 \leq i \leq r$ is at least \begin{equation}
		(1 - \beta) \left( \prod_{i=2}^{k} d_i^{\binom{r}{i}} \right) m^r.
		\label{equation:regularslice-counting}
	\end{equation}
\end{lemma}

\begin{lemma}[Regular Restriction Lemma~{\cite[Lemma 28]{ABCM2017}}] \label{lemma:regularrestriction}
	Suppose $k, r, s, m$ are positive integers and $\zeta, \varepsilon, \varepsilon_k, d_2, \dotsc, d_k > 0$ are such that
	\[ \frac{1}{m} \ll \varepsilon \ll \varepsilon_k, d_2, \dotsc, d_{k-1} \qquad \text{and} \qquad \varepsilon_k \ll \zeta, \frac{1}{k}. \]
	Let $\mathcal{J}$ be an $r$-partite $(k-1)$-complex whose vertex classes $V_1, \dotsc, V_r$ each have size $m$,
	let $G$ be a $k$-graph supported on $\mathcal{J}_{k-1}$,
	and suppose that $G \cup \mathcal{J}$ is a $(d_k, d_{k-1}, \dotsc, d_2, \varepsilon, \varepsilon_k, s)$-regular complex.
	Let $V'_i \subseteq V_i$ be such that $|V'_i| = |V'_j| \geq \zeta m$ for all $1 \leq i< j \leq r$,
	and let $\mathcal{J}' \cup G' = (\mathcal{J} \cup G)[V'_1 \cup \dotsb \cup V'_r]$ be the induced subcomplex.
	Then $\mathcal{J}' \cup G'$ is a $(d_k, d_{k-1}, \dotsc, d_2, \sqrt{\varepsilon}, \sqrt{\varepsilon_k}, s)$-regular complex.
\end{lemma}

\begin{proposition} \label{proposition:uniformlydensematchingviaregularity}
	Suppose $k,r,m,s, t_0, t_1$ are positive integers and $\zeta, \rho, \varepsilon, \varepsilon_k, d_2, \dotsc, d_k > 0$ are such that
	\[ \frac{1}{m} \ll \frac{1}{s}, \varepsilon \ll \varepsilon_k, d_2, \dotsc, d_{k-1}, \qquad \varepsilon_k \ll \zeta, \frac{1}{k}, \frac{1}{r}, \qquad \text{and} \qquad \rho \ll d_2, \dotsc, d_k, \frac{1}{k}, \frac{1}{r} \]
	Let $G$ be a $k$-graph,
	and let $\mathcal{J}$ be a  $(t_0,t_1,\varepsilon,\varepsilon_k,s)$-regular slice for $G$.
	Let $V_1, \dotsc, V_r$ be $r$ clusters of $\mathcal{J}$, each of size $m$, which form a copy of $\clique{k}{r}$ in $R_{d_k}(G)$.
	Then $H[V_1, \dotsc, V_r]$ is $(\zeta, \rho)$-uniformly dense, where $H = K_r(G)$.
\end{proposition}

\begin{proof}
	We need to show that $H[V_1, \dotsc, V_r]$ is $(\zeta, \rho)$-uniformly dense,
	which is equivalent to the following:
	for each $m' \geq \zeta m$,
	and each choice of $W_1, \dotsc, W_r \subseteq V(G)$
	such that for all $1 \leq i \leq r$,
	$|W_i| = m'$ and $W_i \subseteq V_i$,
	we need
	\begin{equation}
		e_H(W_1, \dotsc, W_r) \geq \rho (m')^r.
		\label{equation:uniformlydensematchingviaregularity-key}
	\end{equation}
	From now on, fix $m' \geq \zeta m$ and sets $W_1, \dotsc, W_r$ as before, and let us check \eqref{equation:uniformlydensematchingviaregularity-key} holds for these choices.
	
	Let $\mathcal{J}^r, G^r$ be the restrictions of $\mathcal{J}$ and $G$, respectively, to the clusters $V_1, \dotsc, V_r$.
	Thus $\mathcal{J}^r$ is an $r$-partite $(k-1)$-complex with $r$ vertex classes $V_1, \dotsc, V_r$, each of size $m$,
	and $G^r$ is supported on $\mathcal{J}^r_{k-1}$.
	Since the clusters $V_1, \dotsc, V_r$ form a clique $\clique{k}{r}$ in $R_{d_k}(G)$,
	we deduce that $\mathcal{J}^r \cup G^r$ is a $(d_k, d_{k-1}, \dotsc, d_2, \varepsilon, \varepsilon_k, s)$-regular complex.
	Now, let $\mathcal{J}' \cup G' = (\mathcal{J}^r \cup G^r)[W_1 \cup \dotsb \cup W_r]$ be the induced subcomplex.
	Since $m' \geq \zeta m$, by \cref{lemma:regularrestriction} we deduce that $\mathcal{J}' \cup G'$ is a $(d_k, d_{k-1}, \dotsc, d_2, \sqrt{\varepsilon}, \sqrt{\varepsilon_k}, s)$-regular complex.
	
	Now, apply \cref{lemma:counting} with
	\begin{center}
		\begin{tabular}{c|c|c|c|c|c|c|c|c|c|c|c}
			object/parameter & $k$ & $s$ & $r$ & $m'$ & $d_i$ & $\sqrt{\varepsilon}$ & $\sqrt{\varepsilon_k}$ & $1/2$ & $\mathcal{J}'$ & $G'$ & $W_i$ \\
			\hline
			playing the role of	& $k$ & $s$ & $r$ & $m$ & $d_i$ & $\varepsilon$ & $\varepsilon_k$ & $\beta$ & $\mathcal{J}$ & $G$ & $V_i$ \\
		\end{tabular}
	\end{center}
	for all $1 \leq i \leq r$.
	This yields that the number of copies of $K^k_r$ (with vertices labelled by $1, \dotsc, r$) in $G'$ such that $i \in W_i$ for all $1 \leq i \leq r$ is at least \[ \left[ \frac{1}{2} \prod_{i=2}^{k} d_i^{\binom{r}{i}} \right] (m')^r \geq \rho (m')^r, \]
	where the last inequality follows from the choice of $\rho$.
	Since $G' \subseteq G$,
	and each $K^k_r$ of this form corresponds to an edge in $H[W_1, \dotsc, W_r]$,
	we deduce that $e_H(W_1, \dotsc, W_r) \geq \rho (m')^r$.
	This shows \eqref{equation:uniformlydensematchingviaregularity-key} and finishes the proof.
\end{proof}

With these tools at hand we can finally show \cref{lemma:uniformlydensematching},
using the sketch outlined at the beginning of Section~\ref{section:uniformlydense}.

\begin{proof}[Proof of \cref{lemma:uniformlydensematching}]
	We begin by setting the necessary constants and functions.
	We are given $\alpha, \delta, \varepsilon, k, r$.
	Let $n_1, a_1$ correspond to the values of $n, a$ obtained from \cref{lemma:Hfactoravoiding} with inputs $\alpha/2, r, \delta/2$ in place of $\gamma, r, b$. Define $d_k = \alpha/8$ and $t_0 = n_1$, and let $\varepsilon_1$ be small enough so that Proposition~\ref{proposition:uniformlydensematchingviaregularity} holds with $k, r,\varepsilon_1$ and $\varepsilon$ in place of $k,r,\varepsilon_k$ and $\zeta$, respectively.
	Let $\varepsilon_k = \min \{ \varepsilon_1, a_1/2, \alpha/8 \}$.
	Choose functions $s: \mathbb{N} \to \mathbb{N}$ and $\varepsilon: \mathbb{N} \to (0, 1]$ such that, for all $t \geq 0$,
	Proposition~\ref{proposition:uniformlydensematchingviaregularity} is valid with $s(t), \varepsilon(t), \varepsilon_k$ in place of $s, \varepsilon, \varepsilon_k$, and $1/t$ in place of $d_2, \dotsc, d_{k-1}$.
	Let $t_1$ and $n_0$ be given by Lemma~\ref{lemma:regularslices} with the already-chosen $k, t_0, \varepsilon_k, s, \varepsilon$.
	Let $\rho$ be such that Proposition~\ref{proposition:uniformlydensematchingviaregularity} is true with $1/t_1$ playing the role of $d_2, \dotsc, d_{k-1}$ and $k, r, d_k$ as chosen before;
	and let $d = \rho$.
	Let $m$ be such that Proposition~\ref{proposition:uniformlydensematchingviaregularity} is true with $s(t_1), \varepsilon(t_1)$ in place of $s, \varepsilon$.
	Finally, let $n \geq \max \{ n_0 + t_1!, 2 \delta^{-1} t_1!, 8 \alpha^{-1} t_1!, m t_1 + t_1! \}$.
	
	Now, suppose $H$ is as in the statement of the lemma.
	Our task is to show that there exists $H'\subseteq H$ with $|V(H')|\ge (1-\delta)n$ such that $K_{r}(H')$ admits an $(\varepsilon,d)$-uniformly dense matching.
	
	Remove an arbitrary set $V_0 \subseteq V(H)$ of size less than $t_1!$
	to obtain a subgraph $H'' \subseteq H$ whose number of vertices is divisible by $t_1!$.
	Note that the number of vertices of $H''$ is at least $n_0$,
	and that $\delta(H'') \geq \delta(H) - \alpha n/8$.
	This allows us to apply \cref{lemma:regularslices} to $H''$ to obtain a $(k-1)$-complex $\mathcal{J}$ on $V(H'')$ such that $\mathcal{J}$ is a $(t_0, t_1, \varepsilon(t_1), \varepsilon, s(t_1))$-regular slice for $H''$, and for each set $Y$ of $k-1$ clusters of $\mathcal{J}$, we have
	\begin{equation}
		\left|\overline{\deg}(Y;R(H''))-\overline{\deg}(\mathcal{J}_{k-1}[V_Y];H'')\right|<\varepsilon_k
		\label{equation:uniformlydensematching-slice}
	\end{equation}
	
	Let $R = R_d(H'')$ and $I = I(H'')$ be the $d_k$-reduced $k$-graph and the irregularity $k$-graph of $H''$ with respect to $\mathcal{J}$, respectively.
	If $Y$ an arbitrary $(k-1)$-tuple of clusters of $\mathcal{J}$, then
	\begin{align*}
		\overline{\deg}(Y;R \cup I)
		& \overset{\eqref{equation:regularslice-dreduced+irregular}}{\geq}
		\overline{\deg}(Y;R(H''))-d_k
		\overset{\eqref{equation:uniformlydensematching-slice}}{\geq}
		\overline{\deg}(\mathcal{J}_{k-1}[V_Y];H'') - \varepsilon_k - d_k \\
		& \overset{d_k, \varepsilon_k \leq \alpha/8}{\geq}
		\overline{\deg}(\mathcal{J}_{k-1}[V_Y];H'') - \frac{\alpha}{4}
		\geq 1-\frac{1}{\binom{r-1}{k-1}+\binom{r-2}{k-2}}+\frac{\alpha}{2} \\
		& \overset{\text{Lemma \ref{lemma:cliquetilingbounds}}}{\geq} t(\clique{k}{r}) +\frac{\alpha}{2}
	\end{align*}
	where the second-to-last inequality follows from $\delta(H'') \geq \delta(H) - \alpha n/8$.
	If $t$ is the number of clusters of $\mathcal{J}$,
	this implies that $\delta(R \cup I) \geq \left( t(\clique{k}{r}) +{\alpha}/{2} \right) \binom{t}{k}$. Moreover, since $\mathcal{J}$ is a regular slice, we have $|I| \leq \varepsilon_k t^k \leq a_1 t^k$.
	By the choice of $a_1$ and $t \geq t_0 \geq n_1$,
	\cref{lemma:Hfactoravoiding} implies that $R$ contains a $\clique{k}{r}$-matching $\mathcal{K}$ covering all but at most $\delta t / 2$ vertices of $R$.
	The choice of $n$ and $t \leq t_1$ implies that each cluster of $\mathcal{J}$ has size at least $m$.
	Since $\rho = d$ and by the choice of $\varepsilon$, Proposition~\ref{proposition:uniformlydensematchingviaregularity} implies that each copy of $\clique{k}{r}$ in $\mathcal{K}$ yields an $r$-tuple of equal-sized clusters $(V_1, \dotsc, V_r)$ which form an $(\varepsilon, d)$-uniformly dense tuple in $K_r(H'')$ (and thus also in $K_r(H)$).
	Thus, if $V(\mathcal{K})$ is the set of vertices in clusters covered by $\mathcal{K}$,
	we deduce that $H[V(\mathcal{K})]$ has a $(\varepsilon, d)$-uniformly dense matching.
	
	To finish the proof, set $H' = H[V(\mathcal{K})]$. Then we have
	\[ |V(H')| \geq |V(H'')| - \frac{\delta}{2} t m = n - |V_0| - \frac{\delta}{2} |V(H')| \geq (1 - \delta)n, \]
	where the first inequality follows because $\mathcal{K}$ covers all but $\delta t / 2$ clusters of $\mathcal{J}$ and each cluster has size $m = |V(H')|/t$,
	and later we used that $|V_0| < t_1! \leq \delta n / 2$.
\end{proof}

\bibliographystyle{amsplain}


\begin{aicauthors}
\begin{authorinfo}[mps]
  Mat\'ias Pavez-Sign\'e\\
  University of Warwick\\
  Coventry, CV4 7AL, UK\\
  matias.pavez-signe\imageat{}warwick\imagedot{}ac\imagedot{}uk \\
  \url{https://sites.google.com/view/mpsigne/home}
\end{authorinfo}
\begin{authorinfo}[nsm]
  Nicol\'as Sanhueza-Matamala\\
  Universidad de Concepci\'on\\
  Concepci\'on, Chile\\
  nsanhuezam\imageat{}udec\imagedot{}cl \\
  \url{https://sanhueza.net/nicolas/}
\end{authorinfo}
\begin{authorinfo}[ms]
  Maya Stein\\
  Universidad de Chile\\
  Santiago, Chile\\
  mstein\imageat{}dim\imagedot{}uchile\imagedot{}cl\\
  \url{https://www.dim.uchile.cl/~mstein/}
\end{authorinfo}
\end{aicauthors}

\end{document}